\documentclass[smallextended]{svjour3}     
\smartqed  
\usepackage[colorlinks=false,urlcolor=blue,citecolor=blue,linkcolor=blue,bookmarks=true,bookmarksopen=false,pdftitle=created by dvipdf,pdfcreator=NM,pdfauthor=Jain,pdfsubject=mor.sty:6.29.04]{hyperref}
\usepackage{amsmath,amstext,amsfonts,amssymb,mathrsfs,euscript}
\usepackage{graphicx,color,fullpage,setspace,subfigure,cases}
\RequirePackage{fix-cm}
\newcommand{\sn}{\sqrt{n}}
\newcommand{\tZ} {\tilde{Z}}
\newcommand{\tY} {\tilde{Y}}
\usepackage{personal}

\begin{document}
\title{A Queueing Model with Independent Arrivals, and its Fluid and Diffusion Limits}
\author{
Harsha Honnappa \and Rahul Jain \and Amy Ward
}
\institute{
Harsha Honnappa \at
Ming Hsieh Department of Electrical Engineering\\
University of Southern California\\
3740 McClintock Ave.\\
Los Angeles CA 90089\\
\email{honnappa@usc.edu}
\and
Rahul Jain \at
Ming Hsieh Department of Electrical Engineering \\
University of Southern California\\
3740 McClintock Ave.\\
Los Angeles CA 90089\\
\email{rahul.jain@usc.edu}
\and
Amy R. Ward \at
Marshall School of Business\\
University of Southern California\\
Bridge Hall, Trousdale Parkway\\
Los Angeles, CA 90089-0809\\
\email{amyward@marshall.usc.edu}
}
\date{\today}
\maketitle
\begin{abstract}
We study a queueing model with ordered arrivals, which can be called the $\D_{(i)}/GI/1$  queue. Here, customers from a fixed, finite, population independently sample a time to arrive from some given distribution $F$, and enter the queue in order of the sampled arrival times. Thus, the arrival times are order statistics, and the inter-arrival times are differences of consecutive ordered statistics. They are served by a single server with independent and identically distributed service times, with general service distribution $G$. The discrete event model is analytically intractable. Thus, we develop fluid and diffusion limits for the performance metrics of the queue. The fluid limit of the queue length is observed to be a reflection of a `fluid netput' process, while the diffusion limit is observed to be a function of a Brownian motion and a Brownian bridge process or `diffusion netput' process. The diffusion limit can be seen as being reflected through the directional derivative of the Skorokhod regulator of the fluid netput process in the direction of the diffusion netput process. We also observe what may be interpreted as a sample path Little's law. Sample path analysis reveals various operating regimes where the diffusion limit switches between a free diffusion, a reflected diffusion process and the zero process, with possible discontinuities during regime switches. The weak convergence results are established in the $M_1$ topology.
\end{abstract}

\section{Introduction} \label{sec:intro}

Most of modern queueing theory is concerned with scenarios where arrival and service processes are stationary and ergodic. That the arrival process is a renewal process with i.i.d. inter-arrival times is a common modeling assumption. This is mathematically convenient as it allows full use of the tools that renewal theory and ergodic theory provide. However, it is not true in some queueing scenarios. For example, in some queueing scenarios, each arriving customer takes an independent decision of when to arrive. When we assume that every arriving customer draws an arrival time from the same distribution, this does not lead to a renewal arrival process. Moreover, such a distribution may only have finite support meaning that the system is transient. This scenario does not fit the standard, single-server models in queueing theory such as $M/M/1$, $M/G/1$, etc.


There has been an interest in developing a theory for non-stationary queues \cite{Ne68,Ma81,Ke82,Ma85,Ha92,MaMa95}. 
However, in almost all of these models the assumption of a non-homogeneous Poisson arrival and service process remains ubiquitous. Recent work in \cite{LiWh2012,LiWh2012b} relaxes these assumptions. However, all these models assume a queueing system that operates forever with an infinite population of customers and (possibly) a steady state (when arrival and service rates are cyclostationary). 

In contrast, many queueing systems serve only a finite number of customers, the queueing system itself may operate only in a finite window of time, or a modeler is interested only in the transient behavior of the system.
Scenarios where such behavior is apparent include queueing outside stores before new product launches, DMV or postal offices, lunch cafeterias etc., some call centers where customers take independent decisions of when to call and service time is finite (8am-5pm, for example), and even emergency departments of hospitals, where day-of-week effects strongly indicate that a manager would want to study the queueing dynamics on a single day. In communication networks, single file transfers such as a video streaming session and packet transmissions over a fixed interval of interest are examples of systems where a modeler may wish to study transient delay distributions. 


In this paper, we study a \textit{transitory} queueing model for such systems. 
Consider $n$ customers who arrive into a single-server queue. Each customer's time of arrival is modeled as an i.i.d. sample from a distribution $F$ (restrictions on $F$ will be stated later), and customers enter the queue in order of the sampled times. Service times are i.i.d. with distribution $G$.  If $X_{(i)}$ is the $i$th order statistic from a sample of size $n$ drawn from $F$ and $\Delta_{(i)}:=(X_{(i)}-X_{(i-1)})$ then, in Kendall's notation, this model can be called the $\Delta_{(i)}/GI/1$ \textit{queueing model}.

The analysis of the discrete event model is quite difficult, in general. For instance, when the service process is Poisson, the Kolmogorov forward equations for the joint distribution of the queue length and cumulative arrival processes can be written down, but there is no easy way to obtain analytical solutions. In this paper, we develop fluid and diffusion approximations to the queue-length process directly as the population size scales to infinity and the service rate is \textit{accelerated} appropriately (to be defined). We also establish a sample path Little's Law that links the limit queue-length and virtual waiting time processes under both fluid and diffusion limits.

To develop the fluid limits, we use the Glivenko-Cantelli theorem and the functional Strong Law of Large Numbers for renewal processes along with the Skorokhod reflection mapping theorem.  We show that the fluid limit of the queue length process switches between `overloaded', `underloaded', and `critically loaded' regimes as time progresses.  The limiting diffusion for the queue-length process is derived using a directional derivative reflection mapping lemma.  The diffusion process approximation is a reflection of a Brownian bridge process, that arises from the invariance principle related to the Kolmogorov-Smirnov statistic, combined with a Brownian motion, that arises from the functional central limit theorem for renewal processes.
We also note that our diffusion process convergence results are in Skorokhod's $M_1$ topology on $\mathcal{D}_{\lim}[0,\infty)$, the space of functions that are right or left continuous at every point, and right continuous at 0. 

The rest of this paper is organized as follows. We start with a brief review of the existing literature related to this model. Section \ref{sec:prelim} presents the $\Delta_{(i)}/GI/1$ queueing model and some basic results about fluid and diffusion approximations to arrival and service processes. Section \ref{sec:fluid} develops fluid approximations to the queue length, busy-time and virtual waiting time processes.  In Section \ref{sec:diffusion}, we develop diffusion approximations to these processes. Section \ref{sec:workload-diffusion} develops waiting time approximations, as well as a sample path Little's law. Section \ref{sec:paths} takes a closer look at the sample paths of the queue length process in various operating regimes. Section \ref{sec:sims} presents some examples and simulations of queue length process. We then conclude in Section \ref{sec:conclusion} with some remarks about potential future directions.  In the appendix, we place proofs that are more technical in nature.

\subsection{Related Literature}

The \emph{form} of the diffusion and fluid approximations to the $\D_{(i)}/GI/1$ queue parallel that of the well studied $M_t/M_t/1$ model in the sense that (1) the fluid limit may switch between overloaded, underloaded, and critically loaded periods, and (2) the diffusion limit arises using a directional derivative for the Skorokhod reflection map. Approximations for the latter model were developed in \cite{MaMa95}, wherein the Poisson arrival and service processes are approximated sample path-wise by Gaussian processes on an accelerated time scale, by leveraging strong approximation results for L$\acute{e}$vy processes. We, instead, prove a weak convergence by utilizing the Skorokhod almost sure representation theorem to establish the desired results. Another important difference is that our fluid and diffusion limits depend on empirical process theory (i.e., the Glivenko-Cantelli and Kolmogorov-Smirnov theorems), whereas such results are not relevant in~\cite{MaMa95}.

There have been earlier attempts to understand `transitory' behavior in queueing systems.  In the late 1960s \cite{Ne68} (also \cite{Ne82}), Newell introduced queueing models with both time-varying arrival and service processes. He studied the Fokker-Planck (or heat) equation for the Gaussian process approximation to a general arrival process in various special cases on the arrival rate function. However, these approximations were not rigorously justified with a weak convergence result. In \cite{GaLePe75}, the authors discuss several transitory demand queueing problems and propose a model similar to a $\Delta_{(i)}/M/1$ queue. In \cite{Lo94}, the author considers a similar model to the $\D_{(i)}/GI/1$ queue. The analysis focuses on the local behavior of the queue, similar to the analyses of Newell \cite{Ne68}. The author only establishes local weak convergence to Gaussian processes at continuity points of the limit process. Our results, on the other hand, establish a single ``process-level'' convergence result over all time and, indeed, this is the main difficulty in the analysis. 

\section{Preliminaries} \label{sec:prelim}
\noindent \textbf{Notations.}~~~Unless noted otherwise, all intervals of time are subsets of $[-T_0,\infty)$, for a given $-T_0 \leq 0$ (where $-T_0$ represents the time the first instant a user can arrive; without loss of generality we assume that service starts at $0$). Let $\sD_{\lim} := \sD_{\lim}[-T_0,\infty)$ be the space of functions $x : [-T_0,\infty) \rightarrow \mathbb{R}$ that are right-continuous at $-T_0$, and are either right or left continuous at every point $t > -T_0$. Note that this differs from the usual definition of the space $\sD$ as the space of functions that are right continuous with left limits (cadl\'{a}g functions). We denote almost sure convergence by $\stackrel{a.s.}{\longrightarrow}$ and weak convergence by $\Rightarrow$. $(S,m)$ represents the metric space and metric of convergence. 
Thus, $X_n \stackrel{a.s.}{\longrightarrow} X$ in $(\sD_{\lim},J_1)$ as $n \rightarrow \infty$ indicates that $X_n \in \sD_{\lim}$ converges to $X \in \sD_{\lim}$ in the (strong) $J_1$ topology almost surely. Similarly, $X_n \Rightarrow X$ in $(\sD_{\lim},J_1)$ as $n \rightarrow \infty$ indicates that $X_n \in \sD_{\lim}$ converges weakly to $X \in \sD_{\lim}$ in the (strong) $J_1$ topology. $(\sD_{\lim},M_1)$ indicates that the topology of convergence is the $M_1$ topology. When convergence is joint for a collection of random variables we will either be working with \emph{strong} $M_1$ ($SM_1$) topology or the \emph{weak} $J_1$ ($WJ_1$) topology on the product space of the sample paths (see \cite{Wh01} for formal definitions of these spaces). $\bar{X}$ indicates a fluid-scaled or fluid limit process. $\hat{X}$ and $\tilde{X}$ are used to indicate diffusion-scaled and diffusion limit processes. We use $\circ$ to denote the composition of functions or processes. The indicator function is denoted by $\mathbf{1}_{ \{ \cdot \} }$ and the positive part operator by $(\cdot)_+$. 
\subsection{The Queueing Model} \label{sec:model}
Consider a single server, infinite buffer queue that is non-preemptive, non-idling, and starts empty. Service follows a first-come-first-served (FCFS) schedule. Let $n$ be the customer population size. Customers independently sample an arrival time $T_i$, $i = 1,\ldots,n$, from a common distribution function $F$ assumed to have support $[-T_0,T] \subset \mathbb{R}$, where $T > 0$. For simplicity, we assume that $F$ is absolutely continuous with a continuous density function. The customer entry times are the order statistics $T_{(1)} \leq T_{(2)} \leq \ldots \leq T_{(n)}$ of the sampled arrival times. The arrival process is the cumulative number of customers that have arrived by time $t$:
\begin{equation} \label{def:arrival-process}
A(t) := \sum_{i=1}^n \mathbf{1}_{\{T_i \leq t\}},
\end{equation}
where $\mathbf 1_{\{ \cdot\}}$ represents an indicator function. 

Let $\{\nu_i, i \geq 1\}$ be a sequence of independent and identically distributed (IID) random variables,
where $\nu_i$ represents the service time of the $i$th customer.  Assume that the mean service time $\bbE\nu_i = 1/\mu < \infty$ and the variance of the service times $\mbox{Var}(\nu_i) < \infty$, and that the associated CDF $G$ has support $[0,\infty)$. Finally, also assume that the sequence is independent of the arrival times $T_i$, $i=1,\ldots,n$. Thus, service starts at time $t = 0$. Let $S$ be the service process, defined as a renewal counting process, so that
\begin{equation} \label{def:service-process}
S(t) := \begin{cases} 
0 & \forall t \in [-T_0,0), \\
\sup \{m \geq 1 | V(m) \leq t\}, &~\forall t \geq 0,
\end{cases}
\end{equation}
where 
\[
V(m) := \sum_{i=1}^m \nu_i
\]
is the cumulative load from $m$ jobs. Let $V(t) := \sum_{i=1}^{\lfloor t \rfloor} \nu_i$ be the offered load process.

The amount of time a customer arriving at time $t$ has to wait for service, is
\begin{equation} \label{def:workload-process}
Z(t) := V(A(t)) - B(t) - t \mathbf{1}_{\{t \leq 0\}},
\end{equation}
where 
\begin{equation} \label{def:busy-time}
B(t) := \bigg( \int_{0}^t \mathbf{1}_{\{Q(s) > 0\}} ds \bigg) \mathbf{1}_{\{t \geq 0\}}, ~\forall t \in [-T_0,\infty)
\end{equation}
is the \textit{busy time} process. 
Note that this definition of the virtual waiting time varies slightly from the standard definition due to the fact that an arrival at time $t < 0$ before service starts has to wait an extra $t$ units of time for service to start, which accounts for the $-t \mathbf{1}_{\{t \leq 0\}}$ term.

Let $Q$ represent the queue length process, including both any customer in service and all waiting customers.  This is defined in terms of the arrival and service processes as
\begin{equation} \label{def:queue-length}
Q(t) := A(t) - S(B(t)), \quad \forall t \in [-T_0,\infty),
\end{equation}
where $B(t)$ is the busy time process. 

Finally, the idle time process of the server is
\begin{equation}
\label{idle}
I(t) := t \mathbf{1}_{\{t \geq 0\}} - B(t) =  \bigg (\int_{0}^t \mathbf{1}_{\{Q(s) = 0\}} ds \bigg )
\mathbf{1}_{\{ t \geq 0 \}} \quad \forall t \in [-T_0,\infty).
\end{equation}


\subsection{Basic results}\label{sec:basic}
We now present known functional strong law of large numbers (FSLLN) and functional central limit theorem (FCLT)
or diffusion limits, for the arrival and service processes, as the population size $n$ increases to $\infty$. 

Let $A^n := A$ be the arrival process associated with the system having population size $n$. The fluid-scaled arrival process is
\(
\bar{A}^n := \frac{A^n}{n}.
\)
Next, consider an \textit{accelerated} service process, where the service times (or, equivalently, the service rate) are scaled by the population size $n$, so that 
\[
S^n(t) := \begin{cases}
0 & \forall t \in [-T_0,0),\\
\sup \bigg\{ m \geq 1 | \sum_{i=1}^m \frac{\nu_i}{n} \leq t \bigg \}, &~\forall t \geq 0.
\end{cases}
\]
The fluid-scaled service process is
\(
\bar{S}^n := \frac{S^n}{n}.
\)
Also, the fluid-scaled offered load process is
\begin{equation} \label{offered-work}
\bar{V}^n(t) := \begin{cases}
0 & \forall t \in [-T_0,0),\\
\sum_{i=1}^{\lfloor nt \rfloor} \nu^n_i, & \forall t \in [0,\infty).
\end{cases}
\end{equation}
Note that our assumption that ${\nu_i, i \geq 1}$ is an i.i.d. sequence implies that $S^n(t)$ is equivalent to the time-scaled process $S(nt)$ (where $n$ is an arbitrary parameter that increases to infinity) used in the conventional heavy-traffic setting. Acceleration, however, provides a nice interpretation to our scaling that we conjecture can potentially be extended to non-i.i.d. settings. Now, the following proposition establishes the fluid limits for these processes.

\begin{proposition} \label{prop:fluid-1}
~As $n \rightarrow \infty$,
\begin{eqnarray} \label{lim:fluid-ASV}
(\bar{A}^n(t),\bar{S}^n(t) \mathbf{1}_{t \geq 0},\bar{V}^n(t) \mathbf{1}_{t \geq 0})
\stackrel{a.s.}{\longrightarrow} (F(t), \mu t \mathbf{1}_{\{t \geq 0\}},
\frac{t}{\mu} \mathbf{1}_{\{t \geq 0\}}) \text{ in } (\sD^3_{\lim}, WJ_1),
\end{eqnarray}
where $\sD_{\lim}^3$ is the three dimensional product space of sample paths.
\end{proposition}
\noindent \textbf{Remarks.} 1. The proof of Proposition  \ref{prop:fluid-1} follows easily from standard results and we omit it. The fluid arrival process limit is given by the Glivenko-Cantelli Theorem (see \cite{Du10}). The fluid limits of the service process and the offered load process follow from the functional strong law of large numbers for renewal processes (see \cite{ChYa01}). Joint convergence is a consequence of the independence assumptions between the service times and arrival times.\vspace{10 pt}

Next, looking at the errors of the fluid-scaled arrival process around the fluid limit, the diffusion-scaled arrival process is 
\[
\hat{A}^n(t) := \sn \bigg( \bar{A}^n(t) - F(t) \bigg) \quad \forall t \in [-T_0,\infty).
\]
Similarly, the diffusion-scaled service and offered load processes are 
\begin{eqnarray*}
\hat{S}^n(t) &:=& \sn \bigg( \bar{S}^n(t) - \mu t\bigg), \quad t \geq 0 \\
\hat{V}^n(t) &:=&  \sn \bigg(  \bar{V}^n(t) - \frac{1}{\mu}t \bigg), \quad t \geq 0.
\end{eqnarray*}
The following proposition presents the diffusion limits for these processes.

\begin{proposition}
\label{prop:diffusion-1}
~As $n \rightarrow \infty$,
\begin{eqnarray}
\label{lim:diffusion-ASV}
(\hat{A}^n,\hat{S}^n,\hat{V}^n) \Rightarrow \bigg( W^0 \circ F, \sigma
\mu^{3/2} W \circ e, -\sigma \mu^{1/2} W \circ \frac{e}{\mu} \bigg) \text{ in } (\sD_{\lim}^3,WJ_1),
\end{eqnarray}
where $W^{0}$ is the standard Brownian bridge process and $W$ is the standard Brownian motion process, both are mutually independent, and $e: [0,\infty) \rightarrow [0,\infty)$ is the identity map.
\end{proposition}

\noindent \textbf{Remarks.} 1. The proof of this proposition follows easily from standard results: The FCLT limit for the  diffusion-scaled arrival process, also called the empirical process, is a Brownian bridge by Donsker's Theorem (see Sections 13 and 16 in \cite{Bi68}). Note that this limit also arises in the study of the invariance principle associated with the Kolmogorov-Smirnov statistic used to compare empirical distributions with candidate ones (see \cite{Wh01} for more detail). The limits for the diffusion-scaled service and offered work processes follow from the FCLT for renewal processes (see Section 16 in \cite{Bi68} and Chapter 5 in \cite{ChYa01}). Joint convergence follows from independence.\vspace{3 pt}

\noindent 2. Our assumption that the support of $F$ is compact is largely for technical reasons; viz., the Skorokhod topologies restrict weak convergence to compact intervals of the domain $[-T_0,\infty)$. Proving a diffusion approximation that holds for distributions with infinite support would require strong approximation results, and is beyond the scope of the current paper.

\section{Fluid Approximations} \label{sec:fluid}

Following \eqref{def:queue-length} the fluid-scaled queue length process is
\begin{equation} \label{queue-length-fluid-scaled}
\frac{Q^n(t)}{n} \,=\, \frac{1}{n} A^n(t) - \frac{1}{n} S^n(B^n(t)),
\end{equation}
where $B^n(t)$ is the fluid-scaled version of the busy time process \eqref{def:busy-time} defined as
\[
B^n(t) \,:=\, \bigg ( \int_{0}^t \mathbf{1}_{\{Q^n(s) > 0\}} ds \bigg )
\mathbf{1}_{\{ t \geq 0 \}}.
\]
Next, add and subtract the functions $F(t)$, $\mu t \mathbf{1}_{\{t \geq 0\}}$ and $\mu B^n(t)$ to obtain
\begin{equation*}
\frac{Q^n(t)}{n} := \bigg( \frac{A^n(t)}{n} - F(t) \bigg) - \bigg(\frac{S^n(B^n(t)}{n} -
\mu B^n(t) \bigg) + \bigg( F(t) - \mu t \mathbf{1}_{\{t \geq 0\}} \bigg) + \mu I^n(t),
\end{equation*}
where 
\(
I^n(t)=t\mathbf{1}_{\{t \geq 0\}}-B^n(t)
 \)
is the fluid-scaled idle time process. Thus, \eqref{queue-length-fluid-scaled} is equivalently
\begin{equation}
\label{queue-length-bar}
\overline{Q}^n(t) := \frac{Q^n(t)}{n} \,=\, \bar{X}^n(t) + \mu I^n(t), \quad \forall t \in
[-T_0,\infty),
\end{equation}
where $\bar{X}^n(t)$ is 
\begin{equation}
\label{X-bar-n}
\bar{X}^n(t) := \bigg( \frac{A^n(t)}{n} - F(t) \bigg) - \bigg( \frac{S^n(B^n(t))}{n}
- \mu B^n(t) \bigg) + (F(t) -
\mu t \mathbf{1}_{\{t \geq 0\}}).
\end{equation}

In preparation for the main Theorem in this section, recall that the Skorokhod reflection map is a continuous functional $(\Phi, \Psi) : \sD_{\lim} \to \sD_{\lim} \times \sD_{\lim}$ defined as
\(
x \mapsto \Psi(x) := \sup_{-T_0 \leq s \leq t} (-x(s))_+,
\)
and
\(
x \mapsto \Phi(x) := x + \Psi(x), \quad \forall x \in \sD_{\lim}.
\)
The continuity of the map with respect to the uniform topology on $\sD_{\lim}$ follows from Theorem 3.1 in \cite{MaRa10}. 

\begin{theorem} [Fluid Limit]
\label{thm:queue-length-fluid}
~The pair $(\bar{Q}^n, \mu I^n)$ has a unique representation $(\Phi(\bar{X}^n),\Psi(\bar{X}^n))$ in terms of $\bar{X}^n$. Furthermore, as $n \rightarrow \infty$,
\[
(\bar{Q}^n, \mu I^n) \stackrel{a.s.}{\longrightarrow} (\Phi(\bar{X}),
\Psi(\bar{X}))~ \text{ in } (\sD_{\lim}\times \sD_{\lim},WJ_1),
\]
where $\bar{X}(t) = (F(t) - \mu t \mathbf{1}_{\{t \geq 0\}})$.
\end{theorem}
\Proof
First note that $\bar{Q}^n(t) \geq 0, ~ \forall t \in  [-T_0,\infty)$.  It is also true that $I^n(-T_0)  = 0$ and $d I^n(t) \geq 0, ~ \forall t \in [-T_0,\infty)$. By definition of $I^n(t)$, it follows that  $\int_{-T_0}^{\infty} \bar{Q}^n(t) d I^n(t) = 0$. Thus, by  the Skorokhod reflection mapping theorem (first proved in \cite{Sk56}), the joint process ($\bar{Q}^n(t)$, $\mu I^n(t)$) has a unique reflection mapping representation in terms of $\bar{X}^n(t)$ as \(  (\Phi(\bar{X}^n), \Psi(\bar{X}^n)). \)

Note that by definition of $B^n(t) \leq t$ and from Proposition  \ref{prop:fluid-1}, it follows that
\(
\bigg ( \frac{S^n \circ B^n}{n} - \mu B^n \bigg )
\stackrel{a.s.}{\longrightarrow} 0~ \text{ in } (\sD_{\lim},J_1).
\)
Using this and Proposition  \ref{prop:fluid-1} it follows that
\(
\bar{X}^n \stackrel{a.s.}{\longrightarrow} \bar{X}~ \text{ in } (\sD_{\lim},J_1),
\)
where $\bar{X} := (F(t) - \mu t \mathbf{1}_{\{t  \geq 0\}})$. Using the limit derived above and the continuous mapping theorem, it follows that
\[
(\bar{Q}^n, \mu I^n) = (\Phi(\bar{X}^n), \Psi(\bar{X}^n)) \stackrel{a.s.}{\longrightarrow} (\Phi(\bar{X}),
\Psi(\bar{X}))~\text{ in } (\sD_{\lim} \times \sD_{\lim}, WJ_1).
\]
\EndProof
\noindent \textbf{Remarks.} 1. $\bar{X}$ is the difference between the fluid limits of the arrival and service processes, and is often referred to as the fluid limit of the \textit{netput} process. \vspace{3 pt}

\noindent 2. Theorem \ref{thm:queue-length-fluid} shows that the fluid limit of the queue length process is
\begin{eqnarray*}
\bar{Q}(t) \,=\, (F(t) - \mu t \mathbf{1}_{\{ t \geq 0\}})
+ \sup_{-T_0 \leq s \leq t} (-(F(s) - \mu s \mathbf{1}_{\{s \geq
  0\}}))_{+}, ~ \forall t \in [-T_0,\infty).
\end{eqnarray*}
$\bar{Q}$ can be interpreted as the sum of the fluid netput process and the amount of fluid service capacity lost from the system. As it will be seen below, the time instants where the regulator term $\sup_{-T_0 \leq s \leq t} (-(F(s) - \mu s \mathbf{1}_{\{s \geq
  0\}}))_{+}$ increases are precisely where the queue idles. \vspace{3 pt}

\noindent 3. Figure \ref{fig:fluid} depicts an example queue length process in the fluid limit, and its dependence on the arrival distribution $F$ and service rate $\m$. Note that the process switches between being positive and zero, during the time the server operates. We will investigate this behavior in detail in Section \ref{sec:paths}. Without formally defining the terms, intuitively it should be clear that on [$-T_0,\t_0$) and $[\t_2,\t_3)$ the queue is 'overloaded', while on the intervals $[\t_0,\t_1)$ and $[\t_3\infty)$ it is `underloaded'. \vspace{3 pt}


\begin{figure}[t]
\centering
\includegraphics[scale=0.5]{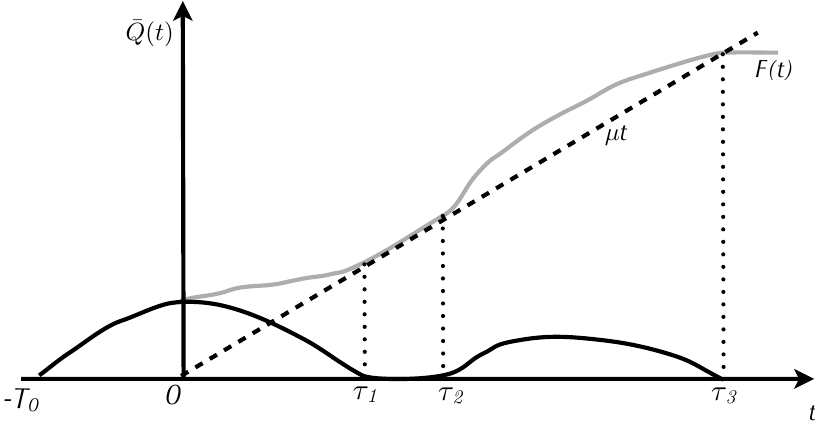}
\caption{An example of a $\D_{(i)}/GI/1$ queue that will undergo multiple
  ``regime changes''. The fluid queue length process is positive on $[-T_0,\tau_0)$ and $[\tau_2,\tau_3)$, and $0$ on
$[\tau_0,\tau_2)$ and $[\tau_3,\infty)$.}
\label{fig:fluid}
\end{figure}

Next, consider the busy time process. It is interesting to observe that $B^n$ does not converge to the identity process, in contrast to the conventional heavy-traffic approximation setting.

\begin{corollary}
\label{cor:busy-time-fluid}
~As $n \rightarrow \infty$,
\begin{equation}
\label{busy-time-fluid}
B^n \stackrel{a.s.}{\longrightarrow} \bar{B}  \text{ in } (\sD_{\lim},J_1)
\end{equation}
where $\bar{B}(t) \,:=\, t \mathbf{1}_{\{t \geq 0\}} -
\frac{1}{\mu} \Psi(\bar{X}(t))$, $\forall t \in [-T_0,\infty)$.
\end{corollary}
\Proof
By definition, we have $B^n(t) = t \mathbf{1}_{\{t \geq 0\}} -
I^n(t)$. This can be rewritten as
\(
B^n(t) = t \mathbf{1}_{\{t \geq 0\}} - I^n(t).
\)
Using Theorem \ref{thm:queue-length-fluid}, the claim then follows.
\EndProof
Note that $\bar{B}(t) = 0$ for all $t \leq 0$, as $\Psi(\bar{X})(t) = 0$ on that interval.

\section{Diffusion Approximations}\label{sec:diffusion}

In this section we assume $F$ is absolutely continuous in order to establish the desired limit result. As noted before, this is mainly for simplicity of the analysis. 

\subsection{Queue Length Process} \label{sec:queue-length-diffusion}

Define the \textit{diffusion-scaled queue length process} as
\begin{equation}
\frac{Q^n(t)}{\sn} \,:=\, \frac{A^n(t)}{\sn} -
\frac{S^n(B^n(t))}{\sn}, ~~\forall t \in [-T_0,\infty)
\end{equation}
Introducing the terms $\sn \mu t
\mathbf{1}_{\{t \geq 0\}}$, $\sn F(t)$ and $\sn \mu B^n(t)$ we have
\begin{equation*}
\begin{split} \frac{Q^n(t)}{\sn} = \bigg( \frac{A^n(t)}{\sn} - \sn
 F(t) \bigg) - &\bigg(\frac{S^n(B^n(t))}{\sn} - \sn \mu B^n(t) \bigg)\\
  &+ \sn (F(t) - \mu t \mathbf{1}_{\{t \geq 0\}}) + \sn \mu (t \mathbf{1}_{\{ t \geq 0\}} - B^n(t)).
  \end{split}
 \end{equation*}
Recalling the definition of the idle time process $Q^n/\sn$ is
\begin{equation} \label{def:queue-length-diffusion-scaled}
\frac{Q^n}{\sn} = \hat{X}^n + \sn \bar{X} + \sn \mu I^n,
\end{equation}
where
\begin{eqnarray}
\label{X-hat-n}
\hat{X}^n(t) \,&:=&\, \bigg ( \frac{A^n(t)}{\sn} - \sn
F(t) \bigg ) - \bigg (\frac{S^n(B^n(t))}{\sn} - \sn \mu B^n(t)  \bigg )\\
\nonumber
&=&\, \hat{A}^n(t) - \hat{S}^n(B^n(t)), \quad \forall t \in [-T_0,\infty).
\end{eqnarray}

Recall from Theorem \ref{thm:queue-length-fluid} that
$\bar{X}(t) =  (F(t) - \mu t \mathbf{1}_{t \geq 0}) $ is the fluid netput process. Lemma \ref{lem:X-hat} below proves a diffusion approximation to the diffusion-scale refinement $\hat{X}^n(t)$ as an immediate consequence of Proposition \ref{prop:diffusion-1}.
\begin{lemma}
\label{lem:X-hat}
~As $n \rightarrow \infty$,
\begin{equation}
\label{X-hat}
\hat{X}^n \Rightarrow \hat{X} \,:=\, W^0 \circ F - \sigma \mu^{3/2}
W \circ \bar{B}~ \text{ in } (\sD_{\lim},J_1)
\end{equation}
where $\bar{B}$ is defined in \eqref{busy-time-fluid}, and $W^0$ and $W$ are independent standard Brownian bridge and standard Brownian motion respectively.
\end{lemma}

\Proof
First note that $B^n(t) \leq t, \forall t \in [0,\infty)$,  implying that $S^n \circ B^n \in \sD_{\lim}$. Using Proposition  \ref{prop:diffusion-1}, Corollary \ref{cor:busy-time-fluid} and the random time change theorem (see, for example, Section 17 of \cite{Bi68}), it follows that
\(
\sn \bigg ( \frac{S^n \circ B^n}{n} - \mu B^n\bigg ) \Rightarrow
\sigma \mu^{3/2} W \circ \bar{B}.
\)
Now, it follows from Proposition \ref{prop:diffusion-1} that
\(
\hat{X}^n \Rightarrow \hat{X}(t) \,:=\, W^0 \circ F  - \sigma
\mu^{3/2} W \circ \bar{B},
\)
thus concluding the proof.
\EndProof

\noindent \textbf{Remarks.} 1. Note that using a classical time change (see, for example, \cite{KaSh91}) it is possible to see that the Brownian bridge is equal in distribution to a time changed Brownian motion, and $\hat{X}$ is equal in distribution to a stochastic integral
\begin{equation} \label{X-hat-integral}
\hat{X}(t) \stackrel{d}{=} \begin{cases} \int_{-T_0}^t \sqrt{g^{'}(s)} d \tilde{W}_s, \quad &\forall t \in [-T_0, T]\\ -\sigma \mu^{3/2} W(\bar{B}(\t^* \vee T)), &\quad \forall t > \t^* \vee T \end{cases},
\end{equation}
where 
\(
g(t) = F(t)(1-F(t)) + \sigma^2 \mu^3 \bar{B}(t),
 \)
 $\tilde{W}$ is a standard Brownian motion process, $\t^* := \frac{1}{\m}$ and $\vee$ is the $\max$ operator. Thus, the process $\hat{X}$ can also be interpreted as  a time-changed Brownian motion on the interval $[-T_0,T]$, and its sample path is a constant on $(T,\infty)$. \vspace{10 pt}

In the rest of this section, we will use Skorokhod's almost sure representation theorem \cite{Sk56,Wh01b}, and replace the random processes above that converge in distribution by those defined on a common probability space that have the same distribution as the original processes and converge almost surely. The requirements for the almost sure representation are mild; it is sufficient that the underlying topological space is Polish (a separable and complete metric space). We note without proof that the space $\sD_{\lim}$, as defined in this paper, is Polish when endowed with the $M_1$ topology. This conclusion follows from \cite{Wh01}. The authors in \cite{MaMa95} also point out that \cite{Po76} has a more general proof of this fact.

We conclude that we can replace the weak
convergence in \eqref{lim:diffusion-ASV} by
\begin{eqnarray*}
(\hat{A}^n,\hat{S}^n,\hat{V}^n)
\stackrel{a.s.}{\longrightarrow} \bigg( W^0 \circ F, \sigma \mu^{3/2} W,-\sigma \mu^{1/2} W \circ \frac{h}{\mu} \bigg)~ \text{ in } (\sD_{\lim},J_1),
\end{eqnarray*}
where abusing notation we use the same letters as our original processes. Thus, Lemma \ref{lem:X-hat} implies that
\[
\hat{X}^n \stackrel{a.s.}{\longrightarrow} \hat{X}~ \text{ in }
(\sD_{\lim},J_1), ~\text{ as } n \to \infty.
\]

The FCLT to the queue length process relies on the directional derivative of the Skorokhod reflection map  $(\Phi, \Psi)$, defined as 
\begin{equation}
  \sup_{\nabla_t^{\bar{X}}}(-y)(t) = \lim_{n \rightarrow \infty} \Psi(\sn x + y)(t) - \sn \Psi(x)(t),
\end{equation}
pointwise in $\sD_{\lim}$, where $x \in \sC$ and $y \in \sC$, and
\(
\nabla_{t}^{x} = \{-T_0 \leq s \leq t | x(s) = -\Psi(x)(t)\},
\)
is a correspondence of points upto time $t$ where the fluid netput process achieves an infimum. We can now state and prove our main limit theorem. Let $\tilde Y^n := \sn \m I^n - \sn \Psi(\bar X)$.
\begin{theorem} [Diffusion Limit]
\label{thm:queue-length-diffusion}
~The pair $(\hat{Q}^n, \tilde{Y}^n)$ has a
unique representation in terms of $\hat{X}^n$ and $\sn \bar{X}$ given by
\(
\bigg ( \Phi(\hat{X}^n + \sn \bar{X}) - \sn \bar{Q},
\Psi(\hat{X}^n + \sn \bar{X})- \sn \Psi(\bar{X}) \bigg ),
\)
 where $\bar{Q} = \bar{X} + \Psi(\bar{X})$ is the fluid limit of the queue
length process. Furthermore, as $n \rightarrow \infty$
\[
(\hat{Q}^n, \tilde{Y}^n) \Rightarrow (\hat{X} + \tY, \tY)~ \text{
  in } (\sD_{\lim} \times \sD_{\lim},SM_1),
\]
where $\hat{X}(t) =
W^0(F(t)) - \sigma \mu^{3/2} W(\bar{B}(t))$, and $\tY(t) = \max_{s\in
  \nabla_{t}^{\bar{X}}} (-\hat{X}(s))$ $\forall t \in
[-T_0,\infty)$, and $SM_1$ is the \emph{strong $M_1$} topology on the product space $\sD_{\lim} \times \sD_{\lim}$.
\end{theorem}
\Proof
First, using \eqref{def:queue-length-diffusion-scaled}, it follows by the Skorokhod reflection mapping theorem that
\begin{eqnarray} \label{skorokhod-diffusion}
\bigg (\frac{Q^n}{\sn}, \sn \mu I^n
\bigg ) \,=\, \bigg ( \Phi(\hat{X}^n + \sn \bar{X}), \Psi(\hat{X}^n + \sn
\bar{X}) \bigg ).
\end{eqnarray}
This implies that
\(
\hat{Q}^n = \frac{Q^n}{\sn} -  \sn \bar{Q} =
\Phi(\hat{X}^n + \sn \bar{X}) - \sn \bar{Q}.
\)
Using the fact that $\bar{Q} =
\bar{X} + \Psi(\bar{X})$ and $\Phi(x)= x + \Psi(x)$ for any $x \in \sD_{\lim}$ it follows that
\begin{eqnarray}
\nonumber
\hat{Q}^n &=& \hat{X}^n + \sn \bar{X} + \Psi(\hat{X}^n + \sn
\bar{X}) - \sn (\bar{X} + \Psi(\bar{X})),\\
\label{centered-queue-length}
 &=& \hat{X}^n + \Psi(\hat{X}^n + \sn \bar{X}) - \sn
 \Psi(\bar{X}).
\end{eqnarray}
Next, from the expression for $\sn \mu I^n$ in
\eqref{skorokhod-diffusion} it follows that
\(
\tilde{Y}^n = \Psi(\hat{X}^n + \sn \bar{X}) - \sn \Psi(\bar{X}),
\)
implying that
\(
\hat{Q}^n = \hat{X}^n + \tY^n.
\)
The limit result now follows by use of the following directional derivative reflection mapping lemma which is adapted from Lemma 5.2 in \cite{MaMa95}, and whose proof can be found in the Appendix.

\begin{lemma}[Directional derivative reflection mapping lemma] \label{lem:mama95}
~Let $x$ and $y$ be real-valued continuous functions on $[0,\infty)$, and $\Psi(z)(t) = \sup_{0 \leq s \leq t} (-z(s))$, for any process $z \in \sD_{\lim}$. Let $\{y_n\} \subset \sD_{\lim}$ be a sequence of functions such that $y_n \stackrel{a.s.}{\rightarrow} y$ as $n \rightarrow \infty$. Then, with respect to Skorokhod's $M_1$ topology,
\(
\ty_n := \Psi(\sn x + y_n) - \sn \Psi(x) \longrightarrow \ty := \sup_{s \in
  \nabla_t^x} (-y(s))
\)
as $n \rightarrow \infty$, where $\nabla_t^x = \{0 \leq s \leq t | x(s) = -\Psi(x)(t)\}$.
\end{lemma}

Observe that $\tY_n$ is exactly in the form of $\ty_n$ defined in the lemma above. Lemma \ref{lem:X-hat} and Lemma \ref{lem:mama95} together imply that
\(
\tY_n \stackrel{a.s.}{\longrightarrow} \tY := \max_{s \in
  \nabla_{\cdot}^{\bar{X}}} (-\hat{X}(s))~ \text{ in } (\sD_{\lim},M_1).
\)
It follows that,
\(
\hat{Q}^n = \hat{X}^n + \tilde{Y}^n \stackrel{a.s.}{\longrightarrow} \hat{X} +
\max_{s \in \nabla_{\cdot}^{\bar{X}}} ( -\hat{X}(s) )~ \text{  in } (\sD_{\lim},M_1).
\)

It remains to prove that $\hat Q^n$ and $\tilde Y^n$ converge jointly in the strong $M_1$, or $SM_1$, topology. Notice that the joint process can be written as 
\[
\left(
\begin{array}{c}
\hat Q^n\\
\tilde Y^n
\end{array}
\right) 
=
\left(
\begin{array}{c}
\hat X^n\\
0
\end{array}
\right)
+ \left(
\begin{array}{c}
\Psi(\hat X^n +\sn \bar X) - \sn \Psi(\bar X)\\
\Psi(\hat X^n +\sn \bar X) - \sn \Psi(\bar X)
\end{array}
\right).
\]
The first term on the right hand side converges to 
\[
\mathbf{\hat{X}} := \left(
\begin{array}{c}
\hat X\\
0
\end{array}
\right)
\]
almost surely in $(\sD_{\lim} \times \sD_{\lim},SM_1)$ by Theorem 12.6.1 of \cite{Wh01}, as $\hat X$ is continuous. The second term converges to 
\[
\mathbf{\tilde{Y}} := \left(
\begin{array}{c}
\tilde Y\\
\tilde Y
\end{array}
\right)
\]
almost surely in $(\sD_{\lim} \times \sD_{\lim},SM_1)$. Now, by definition $\mathbf{\hat X}$ is a continuous process and does not share any discontinuity points with $\mathbf{ \tilde Y}$. Therefore, by Corollary 12.7.1 of \cite{Wh01}, the addition operator is continuous, implying that
\[
\left(
\begin{array}{c}
\hat Q^n\\
\tilde Y^n
\end{array}
\right) 
\stackrel{a.s.}{\to}
\left(
\begin{array}{c}
\hat X + \tilde Y\\
\tilde Y
\end{array}
\right)
\]
in $(\sD_{\lim} \times \sD_{\lim}, SM_1)$. Finally, the weak convergence is a direct implication of the almost sure convergence result, thus concluding the proof.
\EndProof

\noindent \textbf{Remarks.}
1. Observe that the diffusion limit to the queue length process is a function of a Brownian bridge and a Brownian motion. This is significantly different from the usual limits obtained in a heavy-traffic or large population approximation to a single server queue. For instance, in the $G/GI/1$ queue, one would expect a reflected Brownian motion in the heavy-traffic setting.  In \cite{MaMa95} it was shown that the diffusion limit process to the $M_t/M_t/1$ queue is a time changed Brownian motion $W(\int \l (s) ds + \int \m (s) ds)$, where $\l(s)$ is the time inhomogeneous rate of arrival process and $\mu(s)$ is that of the service process, reflected through the directional derivative reflection map used in Lemma \ref{lem:mama95}. There are very few examples of heavy-traffic limits involving a diffusion that is a function of a Brownian bridge and a Brownian  motion process. However, there have been some results in other queueing models where a Brownian bridge arises in the limit. In \cite{PuRe10}, for instance, a Brownian bridge limit arises in the study of a many-server queue in the Halfin-Whitt regime. \vspace{3 pt}

\noindent 2. We noted in the remarks after Theorem \ref{thm:queue-length-fluid} that the fluid limit can change between being positive and zero in the arrival interval for a completely general $F$. One can then expect the diffusion limit to change as well, and switch between being a `free' diffusion, a reflected diffusion and a zero process. This is indeed the case. Figure \ref{fig:diffusion} illustrates this for the example in Figure \ref{fig:fluid}. Note that $\forall t \in [-T_0,\t_1)$ $\Psi(\bar{X})(t) = -\bar{X}(-T_0)$, implying that the set $\nabla_t^{\bar{X}}$ is a singleton. On the other hand, at $\t_1$ $\nabla_t^{\bar{X}} = \{-T_0,\t_1\}$. For $t \in (\t_1,\t_2]$, $\Psi(\bar{X})(t) = 0 = \bar{X}(t)$, implying that $\nabla_t^{\bar{X}} = (\t_1,t]$. On $(\t_2,\t_3)$, $\Psi(\bar{X})(t) = 0$, but $\bar{X}(t) > 0$, so that $\nabla_t^{\bar{X}} = (\t_1,\t_2]$. Finally, the fluid queue length becomes zero when the fluid service process exceeds the fluid arrival process in $[\t_3,\infty)$, implying that $\Psi(\bar{X})(t) = -(F(t) - \mu t) > 0$. It can be seen that $\nabla_t^{\bar{X}} = \{t\}$ in this case. \vspace{3 pt}


\begin{figure}[t]
\centering
\includegraphics[scale=0.6]{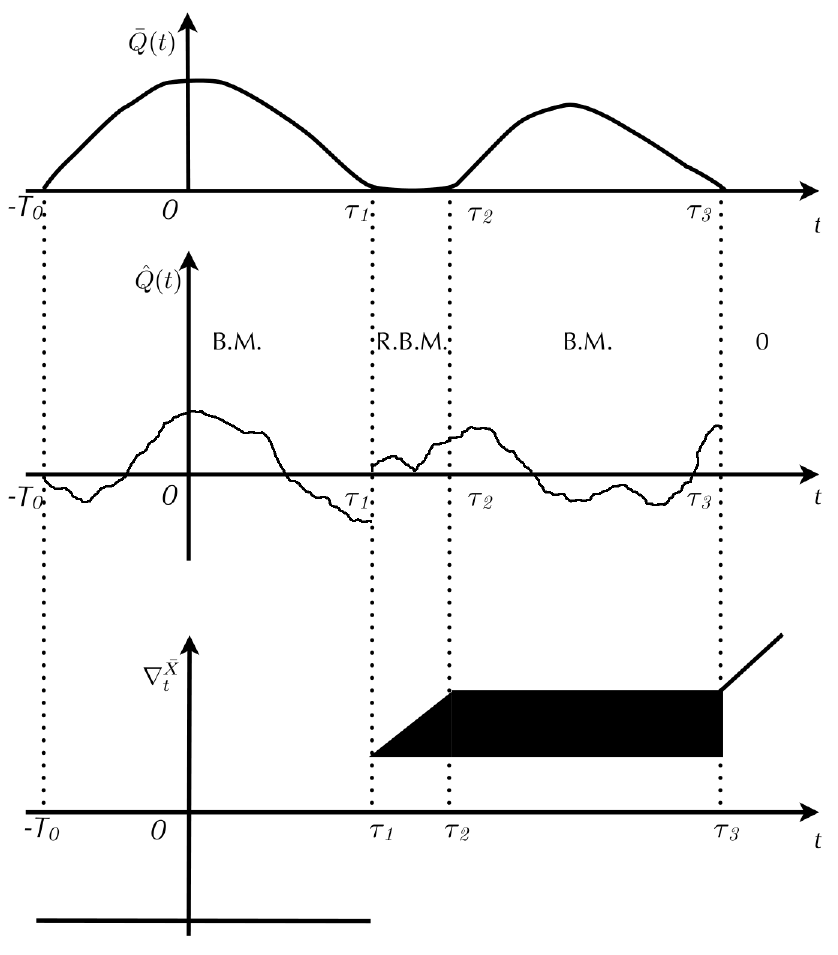}
\caption{An example of a $\D_{(i)}/GI/1$ queue that will undergo multiple
  ``regime changes''. The diffusion limit switches between a free Brownian motion (BM), a reflected Brownian motion (RBM), and the zero process.}
\label{fig:diffusion}
\end{figure}

Recall from Corollary \ref{cor:busy-time-fluid} that $B^n$ converges to a continuous process $\bar{B}$ as $n \to \infty$. Define the diffusion-scaled busy time process as
\begin{equation}
\label{busy-time-diffusion-scale}
\hat{B}^n \,:=\, \sn (\bar{B} - B^n).
\end{equation}
Note that from the definitions of $B^n(t)$ and $\bar{B}(t)$ it follows that
$\hat{B}^n(t) = 0, ~ \forall t < 0$. The diffusion limit for this process is given as follows.

\begin{corollary}
\label{cor:busy-time-diffusion}
The diffusion scaled busy time weakly converges to a regulated diffusion process:
\(
\hat{B}^n \Rightarrow \hat{B} \,:=\, \frac{1}{\mu} \max_{s \in
  \nabla_{\cdot}^{\bar{X}}} (-\hat{X}(s)), ~\text{ in }~ (\sD_{\lim},M_1).
\)
as $n \to \infty$.
\end{corollary}
\Proof
Recall that
\(
B^n(t) = t \mathbf{1}_{\{t \geq 0\}} - I^n(t).
\)
Substituting this and $\bar{B}$ from \eqref{busy-time-fluid} in the definition of $\hat{B}^n$, and rearranging the expression, we obtain
\(
\hat{B}^n = \frac{1}{\mu} \tilde{Y}^n.
\)
A simple application of Theorem \ref{thm:queue-length-diffusion} then provides the necessary conclusion.
\EndProof

Observe that $B^n(t)$ is approximated in distribution by $\hat{B}$ as
\(
B^n(t) \stackrel{d}{\approx} \bar{B}(t) - \frac{1}{\sn} \hat{B}(t),
\)
where $Y \stackrel{d}{\approx} X$ is defined to be
\(
\bbP(Y \leq x) \approx \bbP(X \leq x),
\)
and the approximation is rigorously supported by an appropriate weak convergence result.

The case of uniform $F$ on $[-T_0,T]$ is instructive and it can be seen that on $[-T_0,\t)$ the queue length in the fluid limit is positive. However, as the server starts at time $0$, the only interesting sub-interval of $[-T_0,\t)$ is $[0,\t)$. Using the appropriate definitions, note that $\bar{B}(t) = t$ and $\hat{B}(t) = 0$ for all $t \in [0,\t)$, implying that $B^n(t) = t$ approximately, though in the non-asymptotic regime $B^n(t)$ may be strictly smaller  than $t$. On the other hand, the fluid queue length is zero  in $(\t,\infty)$ and it follows from definition of $\Psi(\bar{X})$ that $\bar{B}(t) = t - \frac{1}{\mu}(-\bar{X}(t)) = \frac{1}{\m} F(t)$ for $t \in (\t,\infty)$. Substituting this expression together with that of $\hat{B}$, and expanding $\hat{X}$, we see that
\[
B^n(t) \stackrel{d}{\approx} t + \frac{1}{\mu} (\bar{X}(t) + \frac{1}{\sn} \hat{X}(t)) \stackrel{d}{=} \frac{1}{\mu} \bigg( F(t) + \frac{1}{\sn} W^0(F(t)) - \sigma \mu W(F(t)) \bigg),
\]
where the second $\stackrel{d}{=}$ is due to the fact that we used the Brownian motion scaling property. Note that this depends on the arrival distribution $F$ alone. In the fluid limit of the busy time process, we see that $\bar{B}(t) = F(t)/\mu$ which is the fraction of time from the interval $[0,t]$ that the queue has spent serving.

\section{Waiting Time and the Sample path Little's Law} \label{sec:workload-diffusion} 
Little's Law is a fundamental tenet of queueing theory, that provides immediate insight into the operation of a queue. While the standard Little's Law relates averages, in this section we prove a large population asymptotic functional relationship that holds on sample paths of the queue length and workload approximations. One may also view this 'sample path Little's Law' as parallel to a snapshot principle in the conventional heavy-traffic setting.

First, the accelerated or fluid-scaled virtual waiting time process is
\(
Z^n(t) = V^n \bigg (n \bigg (\frac{A^n(t)}{n} \bigg ) \bigg ) -\: B^n(t) - t \mathbf{1}_{\{t \leq 0\}}, ~ \forall t \in
 [-T_0,\infty).
\)

\begin{proposition}[Fluid Little's Law]
\label{thm:workload-fluid}
~The fluid-scale workload process is asymptotically related to the queue length fluid limit as $n \to \infty$:
\(
Z^n \stackrel{a.s.}{\longrightarrow} \bar{Z} := \bar{Q}/\mu - e~ \text{ in } (\sD_{\lim},J_1),
\)
where $e : \bbR \to [0,\infty)$ is defined as $e(t) := t \mathbf 1_{\{t \leq 0\}} ~\forall t \in \bbR$ .
\end{proposition}
\Proof
First note that $Z^n(t)$ can be rewritten as
\(
Z^n(t) = V^n \bigg ( n \bigg ( \frac{A^n(t)}{n} \bigg ) \bigg ) - \frac{1}{\mu}
\frac{A^n(t)}{n}
+ \bigg ( \frac{1}{\mu} \frac{A^n(t)}{n} - t
\mathbf{1}_{\{t \leq 0\}} - B^n(t) \bigg ).
\)
Proposition \ref{prop:fluid-1} implies that $\bar{V}^n(t)
\stackrel{a.s.}{\longrightarrow} t/\mu$ in $(\sD_{\lim},J_1)$. Now, using the random time change theorem (Theorem 5.3 in \cite{ChYa01}) and setting $h = A^n/n$ it follows that, as $n \rightarrow \infty$,
\(
\bigg ( V^n \circ A^n - \frac{1}{\mu}\frac{A^n}{n} \bigg )
\stackrel{a.s.}{\longrightarrow} 0~ \text{ in } (\sD_{\lim},J_1).
\)
Using Proposition  \ref{prop:fluid-1} and Corollary \ref{cor:busy-time-fluid}, substituting for $\bar{B}(t)$, we have $\bar{Z}(t) = \frac{1}{\mu} \bar{Q}(t) - t \mathbf{1}_{\{t \leq 0\}}$.
\EndProof

\noindent \textbf{Remarks.} 1. The term $e(t) = t \mathbf{1}_{\{t \leq 0\}}$ accounts for the fact that an arrival at time $t<0$ would require $-t$ time units for service to start.

Now, consider the diffusion-scale virtual waiting time process given by
\(
\hat{Z}^n(t) \,=\, \sn ( Z^n(t) - \bar{Z}(t) ) \quad \forall t \in
[-T_0,\infty).
\)
Proposition \ref{thm:workload-diffusion} below proves a diffusion approximation to $\hat Z^n$ and relates the sample paths of the limit process to that of $\hat Q$.

\begin{proposition}[Diffusion Little's Law]
\label{thm:workload-diffusion}
The diffusion scaled virtual waiting time process satisfies an FCLT in the limit as $n \to \infty$:
\(
\hat{Z}^n \Rightarrow \hat{Z} := \frac{1}{\mu} \hat{Q} + \sigma \mu^{1/2}
W \circ \bar{B} - \sigma \mu^{1/2} W \circ F~ \text{ in } (\sD_{\lim},M_1).
\)
\end{proposition}
\Proof
Expanding the definition of $\hat{Z}^n(t)$ and introducing the term $\frac{1}{\mu}
\frac{A^n(t)}{n}$, we obtain
\(
\hat{Z}^n(t) = \sn \bigg (V^n(A^n(t)) - \frac{1}{\mu}
\frac{A^n(t)}{n} +\frac{1}{\mu} \frac{A^n(t)}{n} -\frac{F(t)}{\mu} + \bar{B}(t) -
B^n(t) \bigg ).
\)
Using the Random Time Change Theorem (Section 17 of
\cite{Bi68}), Proposition  \ref{prop:fluid-1} and Proposition  \ref{prop:diffusion-1}
\begin{eqnarray}
\label{arrival-workload-diffusion}
\sn \bigg ( V^n \circ A^n - \frac{1}{\mu} \frac{A^n}{n} \bigg ) \Rightarrow -\sigma
\mu^{1/2} W \circ \frac{F}{\mu}~ \text{ in } (\sD_{lim},J_1).
\end{eqnarray}
Finally, using this fact, Proposition  \ref{prop:diffusion-1} and Corollary
\ref{cor:busy-time-diffusion}, it follows that
\(
\hat{Z}^n \Rightarrow \hat{Z} =\sigma \mu^{1/2} W \circ \frac{F}{\mu} +
\frac{1}{\mu} W^0 \circ F + \hat{B}~ \text{ in } (\sD_{\lim},M_1).
\)

Note that $W$ and $W^0$ are independent processes. Adding and subtracting the process $\sigma
\mu^{1/2} W \circ \bar{B} $ where $W$ is the Brownian Motion in \eqref{arrival-workload-diffusion}, we obtain
\(
\hat{Z} \,=\, \frac{1}{\mu} \hat{Q} + \bigg( \sigma \mu^{1/2} W \circ \bar{B}
- \sigma \m^{1/2} W \circ \frac{F}{\mu} \bigg).
\)

\EndProof

\noindent \textbf{Remarks.} 1. The limit process in Proposition \ref{thm:workload-diffusion} is equal to
\begin{equation} \label{workload-new}
\hat{Z}(t) = \frac{1}{\m} \hat{Q}(t) - \sigma \m^{1/2} W \bigg (\frac{\bar{Q}(t)}{\m} \bigg ).
\end{equation}
Interestingly, the extra diffusion term is non-zero only when the fluid limit of the queue length process is positive, indicating that it arises from temporal variations in the operating regimes of the queue. To see this, note that the variance of the diffusion term is
\(
\sigma^2 \m~  \bbE \bigg | W(\bar{B}(t)) - W \bigg ( \frac{F(t)}{\m}
\bigg ) \bigg |^2 = \sigma^2 \m
\bigg ( \bar{B}(t) + \frac{F(t)}{\m} - 2 \bar{B}(t) \wedge
\frac{F(t)}{\m} \bigg ),
\)
where $x \wedge y := \min(x,y)$. Clearly, the expression on the right-hand side changes depending upon the ratio of the number of users arrived to the number served in the fluid regime at time $t$. It follows that
\begin{equation*}
\sigma^2 \m~ \bbE \bigg | W(\bar{B}(t)) - W \bigg (\frac{F(t)}{\m}
\bigg ) \bigg |^2  = \begin{cases}
  \sigma^2 \m \bigg ( \frac{F(t)}{\m} - \bar{B}(t) \bigg ), \quad &\frac{F(t)}{\mu
  \bar{B}(t)} > 1 \\
  \sigma^2 \m \bigg ( \bar{B}(t) - \frac{F(t)}{\m} \bigg ), \quad & \frac{F(t)}{\m
  \bar{B}(t)} \leq 1.
\end{cases}
\end{equation*}
\noindent It is easy to see that the first condition above,
\(
F(t)/(\m \bar{B}(t)) > 1,
\)
implies $\bar{Q}(t)/\m > 0$. The second condition,
\(
F(t)/(\m\bar{B}(t)) \leq 1,
\)
implies $\bar{Q}(t) = 0$. This in turn, implies
\(
(F(t) - \mu t \mathbf{1}_{\{t \geq 0\}}) + \Psi(F(t) - \mu t
\mathbf{1}_{\{t \geq 0\}}) = 0.
\)
Rearranging this expression, it follows that $F(t) = \mu t \mathbf{1}_{\{t \geq 0\}} - \Psi(F(t) - \mu t \mathbf{1}_{\{t \geq   0\}})$.

Now, using the definition of $\bar{B}$ from \eqref{busy-time-fluid} we have
\( F(t)/(\m\bar{B}(t)) = 1. \)
It follows that the diffusion term is equal in distribution to the following (time-changed) Brownian Motion
\begin{equation*}
\sigma \mu^{1/2} \bigg ( W(\bar{B}(t)) - W \bigg (\frac{F(t)}{\m}
\bigg ) \bigg )
\stackrel{d}{=} \begin{cases}
\sigma \mu^{1/2} W \bigg (\frac{F(t)}{\m} - \bar{B}(t) \bigg ) = \sigma \m^{1/2}
W \bigg (\frac{\bar{Q}(t)}{\m} \bigg ), \quad &\bar{Q}(t) > 0\\
 \sigma \mu^{1/2} W \bigg (\bar{B}(t) - \frac{F(t)}{\m} \bigg ) = 0, \quad & \bar{Q}(t) = 0.
\end{cases}
\end{equation*}
This  leads to expression \eqref{workload-new}. \vspace{3 pt}

\noindent 2. We note that $\hat Z$ can be interpreted as a sample path Little's Law in the diffusion limit. This result is useful because it provides a sample path relationship between the workload and current queue state. Note that the FCLT of the workload process in a $G/GI/1$ queue (see Chapter 6 of \cite{ChYa01} for details) with arrival rate $\l$ and service rate $\m$ has the form
\(
\tilde{Z}(t) = \frac{1}{\m} \hat{Q}(t) + \sigma \mu^{1/2} (W((\rho \wedge 1) t) - W(\rho t)),
\)
where $\rho = \l/\m$ is the traffic intensity function for the $G/GI/1$ queue, and this is similar to $\hat Z$. The extra diffusion term in \eqref{workload-new} captures the variation of the workload, as the (fluid) queue transitions between various operating states (see Section \ref{sec:paths} for more details on these states). \vspace{3 pt}

\noindent 3. Another interpretation of the term $\sigma \m^{1/2} W(\bar{Q}(t)/\m)$ is that it is in fact the diffusion limit to
the service backlog at time $t$, and the variation in the backlog at
each point in time is captured in the term $\hat{Q}/\m$. Suppose that $f(t) < \mu $ then the fluid queue length process is zero and the server will idle, and the zero state is recurrent for the queue length process. The workload in the system
(for most of the time when $f(t) < \mu $) should be 0. On the other
hand if $F(t) = \mu t$, so that the fluid queue length is zero but the server does not idle, it is reasonable to expect that the virtual waiting time is zero for an arrival at time $t$. However, there is a non-zero
probability of the queue being backlogged at time $t$, and this
fact is captured in the term $\hat{Q}/\mu$. 

\section{Queue Regimes and States} \label{sec:paths}

As noted in Section \ref{sec:diffusion}, the diffusion limit for the
queue length process is piecewise continuous, with discontinuity
points determined by the fluid limit. Indeed, the discontinuity points
are precisely where the fluid limit switches between being
`overloaded' and either `underloaded' and/or `critically-loaded' regimes. We now provide formal definitions of these notions, in terms of the fluid limit arrival and service processes.

We also characterize the sample path of the queue length limit
process, and the points at which it has discontinuities. Developments
in this section follow the study of the directional derivative limit process
in \cite{MaMa95}. However, the limit processes and the setting of our
model are different, as our limit process is a function of a  tied
down Gaussian process while in \cite{MaMa95} the limit process is a
function of a standard Brownian motion. Thus, where necessary, we prove some of
the facts about the sample paths. 

\subsection{Regimes of $\bar{Q}$}\label{sec:regimes}
It is useful to characterize the state of a queue in terms of a
``traffic intensity'' measure. For instance, in the case of a $G/G/1$
queue, the traffic intensity is well defined as the ratio of the arrival rate to the
service rate. This definition is inappropriate for the $\D_{(i)}/GI/1$ queue,
as these systems can be time varying. In \cite{Ma81}, a traffic
intensity function for the $M_t/M_t/1$ queue with arrival rate
$\l(\cdot)$ and service rate $\m(\cdot)$ was introduced as the
continuous function
\[
\rho^*(t) := \sup_{0\leq r \leq t} \frac{\int_r^t \l(u) du}{\int_r^t \m(u) du}, ~t > 0.
\] 
Note that $\rho^*$ follows from the
\textit{pre-limit} model
describing the arrival and service processes in the $M_t/M_t/1$ queue.

For the $\D_{(i)}/GI/1$ queue we define the traffic intensity
in terms of the fluid limit:
\begin{equation} \label{rho}
\rho(t) := \begin{cases}
\infty, \quad & \forall t \in [-T_0,0]\\
\sup_{0 \leq r \leq t}\frac{ F(t) - F(r)}{\mu (t - r)}, \quad &\forall t \in[0,\tilde{T}]\\
0, \quad & \forall t > \tilde{T},
\end{cases}
\end{equation}
where $\tilde{T} := \inf \{t > 0 | F(t) = 1 \text{ and } \bar{Q}(t) =
0\}$. Note that we define the traffic intensity to be $\infty$ in the
interval $[-T_0,0]$ as there is no service, but there can be fluid
arrivals. 
For example, with $F$ uniform
over $[-T_0,T]$, $\rho$ can be shown to be
\[
\rho(t) = \frac{t \wedge T}{t} \frac{1}{\mu(T+T_0)}, ~\forall t \in [0,\tilde{T}].
\]
Note that $\rho$ is continuous in time. Now, consider the following obvious definitions of the operating regimes of the
fluid $\D_{(i)}/GI/1$ queue.

\begin{definition}[Operating regimes.]
~The $\D_{(i)}/GI/1$ queue is (at time $t$)
\begin{enumerate}
\item \textit{overloaded
} if $\rho(t) > 1$.
\item \textit{critically loaded} if $\rho(t) = 1$.
\item \textit{underloaded} if $\rho(t) < 1$.
\end{enumerate}
\end{definition}

The operating regimes can also be referenced in terms of the process
$\bar{Q}$, which in many instances is more intuitive. The following
lemma presents this equivalence.
\begin{lemma} \label{lem:fluid-queue-intensity}
~The $\D_{(i)}/GI/1$ queue is
\begin{enumerate}
\item \emph{overloaded} at time $t$ if $\bar{Q}(t)
  > 0$.
\item \emph{critically loaded} at time $t$ if $\bar{Q}(t) = 0$,
  $\bar{X}(t) = \Psi(\bar{X})(t)$ and there exists an $r < t$ such that
  $\Psi(\bar{X})(t) = \Psi(\bar{X})(s)$ for all $s \in [r,t]$.
\item \emph{underloaded} at time $t$ if $\bar{Q} = 0$, $\bar{X}(t) =
  \Psi(\bar{X})(t)$ and there exists an $r < t$ such that
  $\Psi(\bar{X})(t) > \Psi(\bar{X})(s)$ for all $s \in (r,t)$.
\end{enumerate}
\end{lemma}
The proof of the lemma is in the appendix. Figure \ref{fig:regimes} shows an
example of the various operating regimes with the displayed arrival time
distribution $F$ and service rate $\mu > 1/T$. Here, $BB$ refers to a
Brownian Bridge process and $BM$ refers to a Brownian motion process. Theorem
\ref{thm:queue-length-diffusion} proved a diffusion limit to the
standardized queue length process, and we have shown that
\[
Q^n \stackrel{d}{\approx} L_n \bar{Q} + \sqrt L_n \hat{Q}.
\]
As noted in the remarks
after Theorem \ref{thm:queue-length-diffusion}, the queue length
process switches between being a 'free' diffusion \emph{BB+BM} (when the fluid
limit model is overloaded), to a 'reflected' diffusion \emph{R(BB+BM)} (when the fluid
limit model is critically loaded) and to a 'zero' process \emph{0} (when the
fluid limit model is underloaded).

\begin{figure}[t]
\centering
\includegraphics[scale=0.5]{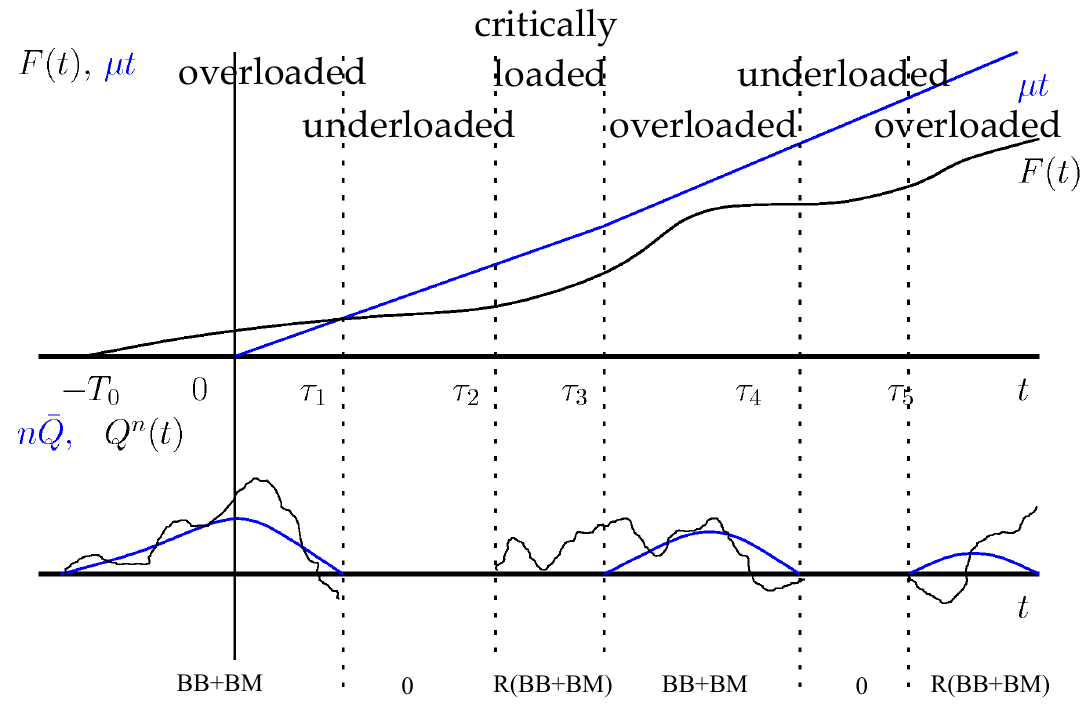}
\caption{An illustration of the various operating regimes of a
  transitory queueing model. Here, we consider the i.i.d. sampling
  $\D_{(i)}$ model.}
\label{fig:regimes}
\end{figure}

Notice that these regimes correspond to those of a
time homogeneous $G/G/1$ queue. However, since the queue length fluid
limit in the $\D_{(i)}/GI/1$ queue can also vary with time we also
identify the following ``finer'' operating states; this is analogous
to the $M_t/M_t/1$ queue as demonstrated in \cite{MaMa95}. In
particular, these states are useful in studying the approximation to
the distribution of queue length process on local time scales. We also
note that Louchard \cite{Lo94} identified some of these operating
regimes in his analysis. The definitions below formalize the intuitive
presentation in \cite{Lo94}.

\begin{definition}[Operating states.] \label{def:granular-states}
~A transitory queue is at
\begin{enumerate}
\item[\emph{(i)}] end of overloading at time $t$ if $\rho(t) = 1$ and there exists an
  open interval $(a,t)$ or $(t,a)$ such that $\rho(r) > 1$ for all $r$ in that interval.
\item[\emph{(ii)}] onset of critical loading at time $t$ if $\rho(t) = 1$ and there
  exists a sequence $\l_n \uparrow t$ such that $\rho(\l_n) < 1$ for all
  $n$.
\item[\emph{(iii)}] end of critical loading at time $t$ if $\rho(t) =1$, and there
  exists a sequence $\l_n \uparrow t$ such that $\rho(\l_n) = 1$ for
  all $n$
  and a sequence $\g_n \downarrow t$ such that $\rho(\g_n) < 1$ for
  all $n$.
\item[\emph{(iv)}] middle of critical loading at time $t$ if $\rho(t) = 1$, and $t$
  is in an open interval $(a,b)$, such that $\sup_{t \in (a,b)}
  \rho(t) \geq 1$ and there exists a sequence $\l_n \uparrow t$ such
  that $\rho(\l_n) = 1$ for all $n$.
\end{enumerate}
\end{definition}

We illustrate how the limit process can be used to approximate the
queue length distribution of the \emph{exact} (pre-limit) model. Our
goal is to study this \emph{distributional approximation} as $\bar{Q}$
and $\hat{Q}$ vary through the various operating regimes and states as
defined above. 

\begin{theorem} [Distributional Approximations] \label{thm:dist-approx}
~The queue length can be approximated in the various operating regimes as follows.\\

\noindent \emph{(i)} Overloaded state.\label{thm:diffusion-overloaded}
Let $t \in (t^*,\t)$ be a time instant of overloading in the
overloaded interval, where $t^* :=
\sup \nabla_t^{\bar{X}}$ and $\t := \inf \{s > t^* | \rho(s) = 1
\}$. Then
\[
\frac{Q^n(t)}{\sn} - \sn (F(t) -
F(t^*) - \m(t - t^*)) \Rightarrow \hat X(t) + X^*, \text{ as } n \to \infty
\]
where $X^* := \sup_{s \in \nabla_{t^*}^{\bar{X}}}
(-\hat{X}(s))$. Further, $\tZ^n_t := \sn (F(t) -
F(t^*) - \m(t - t^*) )+ \hat X(t) + X^*$ is the strong solution to the stochastic differential equation
\(
d \tZ^n_t = \sqrt{n} (f(t) - \m) dt + \sqrt{g^{'}(t)} dW_t ~~\forall t \in (t^*,\t)
\)
with initial condition $\tZ_{t^*} = \hat{X}(t^*) - X^*$, where
and $g(t) = F(t)(1- F(t)) + \sigma^2 \mu^3 \bar{B}(t)$.

\noindent \emph{(ii)} Underloaded state.
If $t$ is a point of underloading, i.e. if $\rho(t) < 1$, then
\(
\frac{Q^n(t)}{\sqrt{n}} \Rightarrow  0,
\)
as $n \to \infty$.

\noindent \emph{(iii)} Middle- and End-of-critically-loaded state.
An open set of the domain $(t^*,\t)$ is a critically loaded interval, where $t^*$ is
a point in the onset of critically loaded state and $\t$ a point at the end of
critically loaded state, as defined in Definition
\ref{def:granular-states}. For any $t \in (t^*,\t)$, let $u = t-t^*$
and we have, as $n \to \infty$,
\[
\frac{Q^n(t)}{\sqrt{n}} \Rightarrow (\hat{X}(t) + \sup_{0 \leq s \leq u} (-
\hat{X}(s))),
 \]
 where $\hat{X}(u) \stackrel{d}{=} \hat{X}(t) - \hat{X}(t^*)$, and
 $\hat{X}(t) \stackrel{d}{=}\int_{-T_0}^t \sqrt{g^{'}(s)} dW_s$.\\

\noindent \emph{(iv)} End of overloading state.
Let $t$ be a point of end of overloading. Then, for all $\t > 0$
\[
\frac{Q^n(t-\frac{\t}{\sqrt{n}})}{\sn} \Rightarrow \left( \hat{X}(t) +
\left( \sup_{s \in \nabla_t^{\bar{X}} \backslash \{t\}} (-\hat{X}(s)) \right) -
(f(t) - \m)\t \right)^+, \text{ as } n \to \infty,
\]
where $f(t)$ is the density function associated with the fluid
limit $F$.
\end{theorem}
The proof is relegated to the appendix.\\

\noindent \textbf{Remarks:} 1. \textit{Overloaded regime.} \noindent
(i) Then the approximate distribution is
Gaussian with mean $F(t) - \mu t$. However, the variance is
affected by the fact that the queue may have idled in the past. Recall
that the variance is $g(t) =
F(t)(1-F(t)) + \sigma^2 \mu^3 \bar{B}(t)$, where from Corollary \ref{cor:busy-time-fluid} 
\[
\bar{B}(t) = \mathbf{1}_{\{t \geq 0\}} - \frac{1}{\m} \Psi(\bar{X})(t).
\]

\noindent (ii) We note that this result is analogous to case 5 of
Section 4 in \cite{Lo94}. However,
in \cite{Lo94}, the author notes that no reflection need be applied in an overloaded
sub-interval, and proceeds to derive the limit process (in this interval
alone) as $W^0 \circ F (t) - \sigma \mu^{3/2} W(t)$. This is not
entirely accurate as the starting state of the process in each new interval of overloading must be factored into the approximation. That is, while
$\nabla_t^{\bar{X}}$ is fixed for all $t$ in an overloaded sub-interval,
the value $\sup_{s \in \nabla_t^{\bar{X}}} (-\hat{X}(s))$ provides
the starting state for the diffusion in such an interval. \vspace{3 pt}

\noindent 2. \textit{Critically-loaded regime.} The queue length process 
in the critically loaded regime is approximated by a driftless reflected
process, with continuous sample paths, with starting state $\hat{X}(t^*)$. By the definition of a critically loaded state $\rho(t) = 1$ at
all such points and $\nabla_t^{\bar{X}}$ ``accumulates'' the
points of critical loading, as $t$ evolves through the critically
loaded interval. It follows that the set $\nabla_t^{\bar{X}}$ is the interval $(t^*,t]$. \vspace{3 pt}

\noindent 3. \textit{End of overloading regime.} As noted in the
definition, a point $t$ is one of end-of-overloading if the traffic
intensity  is 1 at $t$, and is strictly greater than 1 at all points
to the left of it. Here, we are primarily interested in the rate at
which the queue empties out asymptotically as overloading
ends. Consider a sequence of $\t_n$ defined as a sequence of times at
which the queue in the $n$th system first empties out, and define $v := t - \frac{\t_n}{\sn}$. Then, from Theorem \ref{thm:dist-approx}
\[
\t_n = \sn (t - v) \Rightarrow \frac{\hat{X}(t) + \sup_{s \in
    \nabla_t^{\bar{X}}\backslash \{t\}} (-\hat{X}(s))}{ f(t) - \m}
\]
Thus, it can be seen that the time at which the queue empties out
converges to a Gaussian random variable. A similar conclusion was
drawn in \cite{Lo94} and in \cite{MaMa95} for the $M_t/M_t/1$ queue.

\subsection{Sample Paths}\label{sec:sample-paths}

We now characterize a typical sample path of the limit process
$\hat{Q}$.

\begin{proposition} \label{thm:sample-path}
~The process $\hat{Q}$ is upper-semicontinuous almost surely.
\end{proposition}

The following proposition summarizes where discontinuities occur in
$\hat Q$. We note that this is also part of Theorem 3.1 of
\cite{MaMa95}. Since the proof follows that in \cite{MaMa95} we omit it.

\begin{proposition} 
~$\hat Q$ is discontinuous at time $t$, with a nonzero probability, if and only if $t$ is the end-point of overloading or critical loading. The set of such points is nowhere dense.
\end{proposition}

\noindent \textbf{Remarks.} 1. We note that the queue length limit sample paths for the $M_t/M_t/1$ model are also upper-semicontinuous as shown in Theorem 3.1 of \cite{MaMa95}. There the sequence of converging processes was shown to be monotone, which easily leads to upper-semicontinuity by Dini's Theorem. As this monotonicity property does not hold for the corresponding processes in the $\D_{(i)}/GI/1$ model, we argue that the sample path is upper-semicontinuous directly from the characterization of the points of continuity and discontinuity in the domain of the sample path. \vspace{3 pt}

\noindent 2. The intuition for the regime switching behavior proved in
is easy to see in the case of a
uniform arrival distribution with early-bird arrivals, such that the
service rate is greater than the value of the density function function. Here, the (fluid) queue is overloaded on the interval $[-T_0,\t)$ with the singleton set $\nabla_t^{\bar{X}} = \{ -T_0 \}$, and underloaded on the interval $(\t,\infty)$ with the singleton set $\nabla_t^{\bar{X}} = \{t\}$. At $\t$ itself, there are two points in the set $\nabla_t^{\bar{X}} = \{-T_0,\t\}$. Thus, there is a discontinuity due to the fact that the set $\nabla_t^{\bar{X}}$ changes from being a singleton on the interval $[-T_0,\t)$ to $\{-T_0,\t\}$ at $\t$.

 \section{Examples and Simulations} \label{sec:sims}
 We illustrate the queue length process approximations with uniform and exponential arrival time distributions. The former is interesting, as the uniform distribution emerges as the mean field equilibrium arrival profile when arriving users are strategic about when they enter the queue in order to minimize their delay through the queue; see \cite{JaJuSh11,HoJa11}. The exponential distribution case serves to illustrate the fact that many of the conclusions of our theorems can be carried over to infinite support arrival time distributions, though the limit results remain to be fully justified.
\subsection{Uniform Arrival Distribution}
The uniform arrival case is particularly simple and illustrates the discontinuities in the limit processes. Recall that $\nabla_t^{\bar{X}}$ is a correspondence that maps each time $t$ to the set of points (upto $t$) at which the fluid netput process is equal to its infimum at $t$. 

\begin{corollary}
\label{cor:uniform-queue-length-diffusion}
~Let $F$ be the uniform distribution on $[-T_0,T]$, where $-T_0 < 0$. Then,
\[
\hat{Q}(t) =
\begin{cases}
W^0(F(t)) - \sigma \m^{\frac{3}{2}} W(t), \quad &\forall t \in
[-T_0,\t)\\
( W^0(F(\t)) - \sigma \m^{\frac{3}{2}} W(\t)) + ( -(W^0(F(\t)) -
\sigma \m^{\frac{3}{2}} W(\t)))_{+}, \quad &t = \t \\
0, \qquad &\forall t \in (\t,\infty),
\end{cases}
\]
where $\t = \{ -T_0 \leq t < \infty \,|\, F(t) = \m t \}$.
\end{corollary}
\Proof
Recall from Theorem \ref{thm:queue-length-diffusion}
that $\hat{Q} = \hat{X} + \sup_{s \in \nabla_{\cdot}^{\bar{X}}} (
-\hat{X})$ where $\hat{X} = W^0 \circ F - \s
\m^{\frac{3}{2}} W \circ \bar{B}$, and $\bar{B}$ is the fluid busy
time process. Now, using the definition of $\nabla_t^{\bar{X}}$, it is easy to deduce that in this case we have
\[
\nabla_t^{\bar{X}} = \begin{cases}
 \{-T_0\}, \quad &\forall t \in [-T_0,\t),\\
 \{-T_0,\t \}, \quad &t = \t,\\
 \{t\}, \quad \forall &t \in (\t,\infty).
\end{cases}
\]
Further, Corollary \ref{cor:busy-time-fluid} yields
\[
\bar{B}(t) = \begin{cases}
& t, \quad \forall t \in [-T_0,\t],\\
& 0, \quad \forall t \in (\t,\infty).
\end{cases}
\]
Using these facts the conclusion follows by substitution.
\EndProof

The time $\t$ can be interpreted as the first time that the fluid service process catches up
with the fluid arrival process. For a uniform $F$ there is at most one such point, but in general there can be many such points.

\noindent \textbf{Remarks:}
\noindent 1. A useful way to interpret the discontinuity at $\t$ in Corollary \ref{cor:uniform-queue-length-diffusion} is to consider the
process on the two sub-intervals separately and try to ``patch'' them
together. If $\hat{Q}(\t-) = \hat{X}(\t) = \hat{Q}(\t) > 0$ we should
expect a free diffusion path on the interval $[-T_0,\t]$, and a
reflected process such that the path is 0 on $(\t,\infty)$. Furthermore, $\hat{Q}(\t)$
becomes the ``starting state'' for the process on the interval
$(\t,\infty)$, and the reflection operator is applied an instant after
$\t$. On the other hand, if $\hat{Q}(\t-) = \hat{X}(\t-) \leq 0$
we have a free diffusion on $[-T_0,\t)$ and the zero process on
$[\t,\infty)$, i.e., the process drops to zero at $\t$. Thus,
$\hat{Q}(\t-)$ provides the starting conditions for the new ``regime''
of the diffusion, as the process transitions from $[-T_0,\t)$ to
$(\t,\infty)$. \vspace{3 pt}

\noindent 2. We note that in \cite{Lo94}, a diffusion approximation to the queue length process is derived independently for different operating regimes, and as such does not involve the directional derivative reflection map. These limit results have not been ``patched'' together to obtain a ``process-level'' convergence result, which is precisely where the mathematical challenges lie. 
\vspace{3 pt}

Note that the nature of the discontinuity at $\hat{Q}(\t)$ depends on the the sign of $\hat{X}(\t)$. Following \cite{MaMa95} it can be shown $t$ is a point of \emph{right-discontinuity} for a function $x \in \sD_{\lim}$ if $x$ is left-continuous at $t$, and $x(t-) > x(t+)$. On the other hand, $t$ is a point of \emph{left-discontinuity} if $x$ is right-continuous at $t$, and $x(t+) > x(t-)$.

\begin{corollary} \label{cor:uniform-arrivals}
~Let $F$ be the uniform distribution over
$[-T_0,T]$, where $T_0 > 0$, and $\t = \{-T_0\leq t < \infty | F(t) = \mu t
\mathbf{1}_{\{t \geq 0\}} \}$. Then, for the process $\hat{Q}$ in Corollary \ref{cor:uniform-queue-length-diffusion}, we have
\begin{enumerate}
\item[(i)] $[-T_0,\t) \cup (\t,\infty)$ are points of continuity.
\item[(ii)] $\t$ is a point of right-discontinuity, when $\hat{X}(\t)
  \geq 0$.
\item[(iii)] $\t$ is point of left-discontinuity, when $\hat{X}(\t) < 0$.
\end{enumerate}
\end{corollary}
The proof is available in the Appendix.



\begin{figure} [t]
\centering
\subfigure[~Sample queue length process mean for    $n=10,25,100,1000$, averaged over 10000 simulation runs.] {
  \includegraphics[scale=0.55,angle=-90]{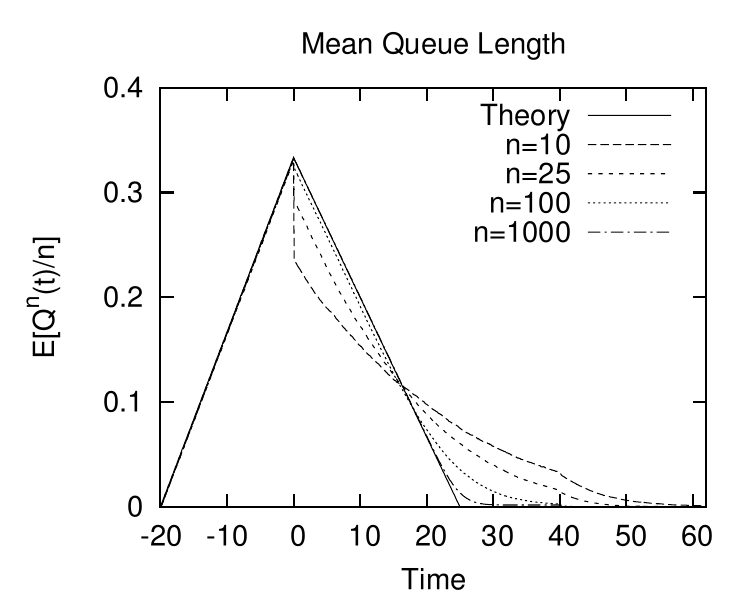}
  \label{fig:mean}
 }
\subfigure[~Sample queue length process variance for $n=10,25,100,1000$, averaged over 10000 simulation runs.] {
  \includegraphics[scale=0.55,angle=-90]{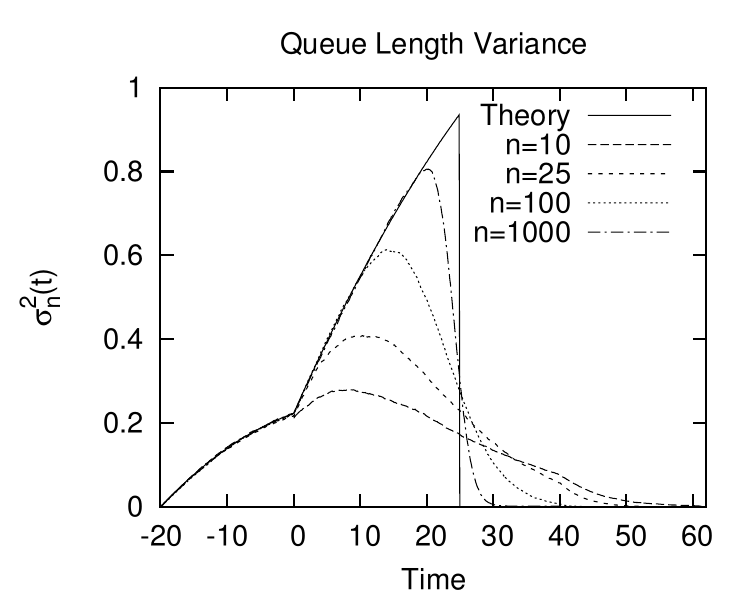}
  \label{fig:variance}
 }
\caption{Typical sample paths, mean and variance envelopes of the queue length process for $F$ uniform over $[-20,40]$, and exponentially distributed service times with rate $\mu = 0.03$.}\end{figure}

Simulations can provide insight into the accuracy of the approximations for various population sizes. Consider a uniform arrival distribution over the interval $[-20,40]$, with service times i.i.d. and exponentially distributed with parameter $\mu = 0.03$. 
Figures \ref{fig:mean} and \ref{fig:variance} show the sample mean and the sample variance of the (scaled) queue length process for $n = 10,25,100,1000$ over 10,000 sample runs. Note that as $n$ increases, the sample mean approaches the fluid limit, and the sample variance approaches the theoretical variance of the queue length process. For the given $F$, the latter quantity is 
\[
\sigma^2(t) \,=\, \begin{cases} F(t)(1-F(t)), \quad &\forall t \in [-T_0,0] \\
F(t)(1-F(t)) + \sigma^2 \mu^3 t, \quad & \forall t \in (0,\tau)\\
0, \quad & \forall t > \tau.
 \end{cases}
\]
Observe from Figure \ref{fig:mean} that even for small $n$, the sample mean is quite close to the fluid limit for $t<0$. However, once queueing dynamics come into play, the fluid limit is a good approximation only for $n=100$ or larger. A similar effect is manifest for the diffusion limit as well: once service starts, and queueing dynamics come into play, the diffusion limit becomes a reasonably good approximation only for $n=1000$ or larger.

\subsection{Exponential Arrival Distribution}

Assume $F$ is an exponential distribution function with parameter $\l > 0$, so that $F(t) = 1 - e^{-\l t}$ and $-T_0 = 0$. Keep in mind that this is unlike the $M/GI/1$ queue where the exponential distribution models the inter-arrival times. Recall that the limit results in Theorems \ref{thm:queue-length-fluid} and \ref{thm:queue-length-diffusion} are proved on compact sets of the domain $[-T_0,\infty)$. Therefore, the limit does not hold simultaneously at all points in the support of $F$, and proving the FSLLN and FCLT for infinite support distributions is beyond the scope of the current paper. However, observe that the queue length fluid model can be conjectured to be
\begin{itemize}
\item[(i)] If $\m \geq \l$, then $\bar Q(t) = 0 ~\forall t \in [0,\infty)$.
\item[(ii)] If $\m < \l$, then
\[
\bar Q(t) = \begin{cases}
(1 - e^{-\l t} - \m t) & \forall t \in [0,\t)\\
0 & \forall t \geq \t,
\end{cases}
\]
\end{itemize}
where $\t := \inf \{t \geq 0 | F(t) = \m t\}$ is the last instant the fluid queue length is positive (also known as the \emph{makespan}). To see this, recall the definition of $\bar Q(t)$ and notice that if $\m \geq \l$ then $\l e^{-\l t} \leq \m, ~\forall t > 0$. This implies that the queue is underloaded, as defined in Section \ref{sec:regimes}.  On the other hand, if $\m < \l$, the system shifts from overload to underload, per our definition in Section \ref{sec:regimes}. It can be shown that $\t = \frac{1}{\l} \mathscr W\left( -\frac{\l}{\m} e^{-\frac{\l}{\m}}\right) + \frac{1}{\m}$, where $\mathscr W(\cdot)$ is the Lambert W-function. To see this, recall that it is the first (strictly positive) solution to $e^{-\l t} = 1 - \m t$. Substituting in $-x = -\l t + \frac{\l}{\m}$, we have $x e^{x} = -\frac{\l}{\m} e^{-\frac{\l}{\m}}$. It is well known that this is the defining equation for the Lambert W function $\mathscr W$, implying that $x = \mathscr W\left( -\frac{\l}{\m} e^{-\frac{\l}{\m}} \right)$. Substituting back for $t$ we obtain the expression for $\t$. 

The fluid model allows us to conjecture the corresponding diffusion refinement. Let $\hat Q$ be the queue length diffusion model. Then,
\begin{itemize}
\item[(i)] If $\m \geq \l$, then $\hat Q(t) = 0 ~\forall t \in [0,\infty)$. 
\item[(ii)] If $\m < \l$, then 
\[
\hat Q(t) = \begin{cases} 
W^0(F(t)) - \sigma \m^{\frac{3}{2}} W(t) & \forall t \in [0,\t)\\
(W^0(F(t)) - \sigma \m^{\frac{3}{2}} W(t)) + (-W^0(F(t)) + \sigma \m^{\frac{3}{2}} W(t))_+ & t = \t\\
0 & \forall t \in (\t,\infty).
\end{cases}
\]
\end{itemize}
The ``proof'' of this is straightforward. Part (i) follows from the fact that the fluid model is underloaded under the same condition. Part (ii) follows from the reasoning in the proof of Corollary \ref{cor:uniform-queue-length-diffusion}. A little algebra shows that the variance cruve of the diffusion approximation $\hat Q$ when $\m < \l$ is given by
\[
\sigma^2(t) = \begin{cases}
F(t) (1-F(t)) + \s^2 \m^3 t & \forall t \in [0,\t),\\
0 & \forall t > \t.
\end{cases}
\]

\begin{figure}[t]
\centering
\subfigure[~Sample mean queue length for $n = 10,25,100,1000$, averaged over 30 simulation runs.] {
  \includegraphics[scale=0.25]{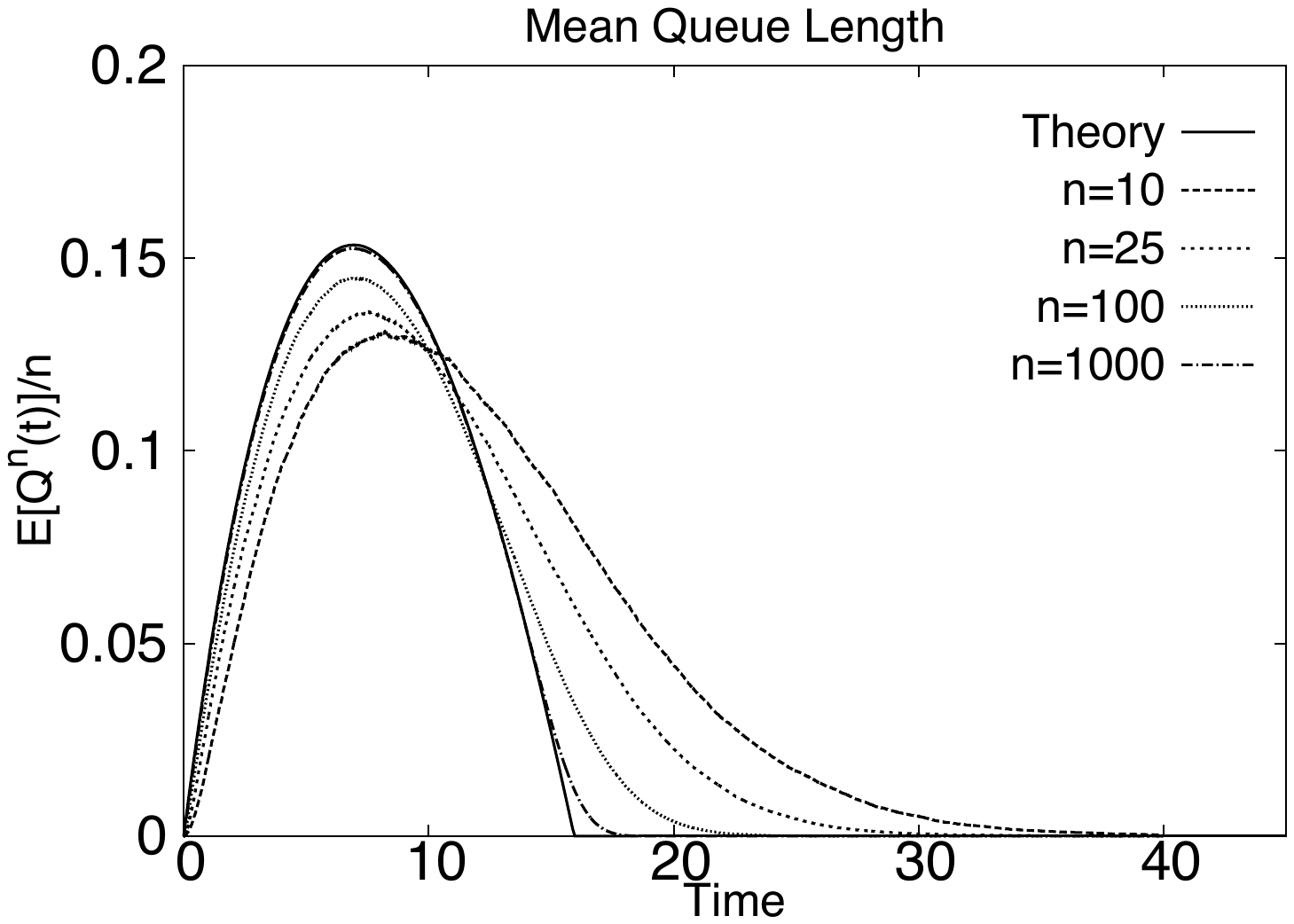}
  \label{fig:exp-dist-b}
}
\subfigure[~Sample variance for $n=10,25,100,1000$, averaged over 30 simulation runs] {
  \includegraphics[scale=0.25]{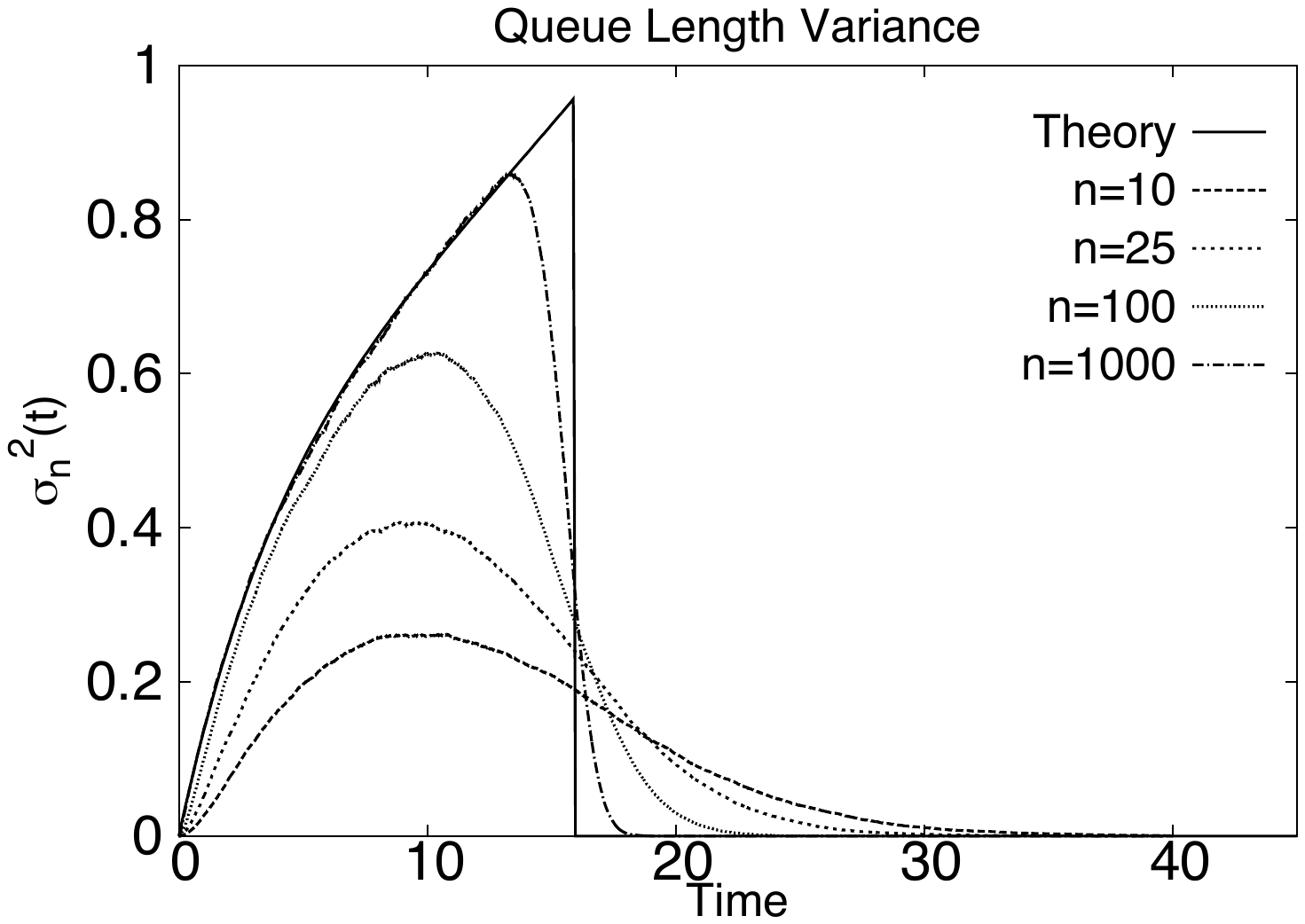}
  \label{fig:exp-dist-c}
}
\caption{Typical sample paths, mean and variance envelopes of the queue length process for $F$ exponential with parameter $\l = 0.1$ and exponentially distributed service times with mean rate $\m = 0.05$.}
\end{figure}

Let us consider a specific example, where $\l = 0.1$ and $\m = 0.05$, in which case it can be verified that $\t = 15.9362$. Figure \ref{fig:exp-dist-b} shows that for even low values of $n$, the fluid limit is a very good approximation to the observed mean queue length. Similarly, Figure \ref{fig:exp-dist-c} shows that the variance of the diffusion limit is a reasonable approximation to the variance of the queue length in the (accelerated) discrete event system. 


We also note a very interesting connection between random graph theory and the $\D_{(i)}/GI/1$ queue, brought to our notice by J.S.H. van Leeuwaarden in a personal communication. Specifically, he has shown that the excursions of the queue length process in the discrete event system, observed at the departure times of jobs, also measures the size of the connected components of a random graph with $n$ vertices. \cite{Al1997} shows that in the ``large graph'' limit (i.e., as $n \to \infty$), the connected components in a (nearly) critical Erd\"{o}s-R\'{e}nyi random graph (see \cite{Du2007} for details on these terms) can be related to the excursions of a Brownian motion on a parabola by a weak convergence limit result linking the two. This type of result is also intimately connected with the question of the final size of an \emph{epidemic} in a critical random graph; see \cite{Ma-Lo1998,VaReLe2010} where the distribution of the final size in a critical Susceptible-Infected-Recovered (SIR) epidemic model is studied. Using a Taylor series expansion on the fluid limit of the queue length, it can be shown that for small $t$ and ignoring terms of order 3 and higher, the diffusion approximation is a Brownian excursion on a parabola. This connection with the $\D_{(i)}/GI/1$ queue might provide a new framework to study the final size distribution of other epidemic models in the critical regime. 

\section{Conclusions and Future Work} \label{sec:conclusion}

In this paper, we introduce a bespoke single server queueing model which we call the $\D_{(i)}/GI/1$ queue, to model systems that are purely transient in nature, and thus serve a finite population of customers. 
We develop pathwise asymptotic fluid and diffusion approximations to the system performance metrics as the population size is increased to infinity. These approximations are unlike the conventional heavy-traffic limits, but are closer in spirit to the uniform acceleration approximations to the $M_t/M_t/1$ queue.

Our original motivation for introducing the $\D_{(i)}/GI/1$ model came from the `concert arrival game', a game of arrival timing introduced in \cite{JaJuSh11}. Customers choose to arrive at a queue to minimize a linear cost functional that depends on the waiting time and the number of people who have already arrived. In the fluid limit, the Nash equilibrium arrival profile was shown to be a uniform distribution function. An important question of interest is whether the equilibrium derived from the fluid model approximates in any way the equilibrium of the finite population `concert arrival game'. 
Our next step is to take the diffusion approximations for the $\D_{(i)}/GI/1$ queue model, and revisit the `concert arrival game' problem. In \cite{JaJuSh11}, the assumption is that the queue lengths are unobservable. Our diffusion approximations can now allow us to study other situations where the queue length are fully or partially observable. In the spirit of mean field game theory, this could be understood to be a `diffusion field game theory'.


An important direction to take this research would be to study transitory queueing models with non-stationary service processes. For instance, customers arriving closer to the end of day may experience shorter service times. We conjecture that the limit results will be interesting but non-trivial to establish.

Finally, it would be interesting to test empirically for how to fit the distribution F, that characterizes the arrival pattern, to data. Then, it would be possible to use the wait time predictions suggested by the $\Delta_{(i)}/GI/1$ model to make capacity sizing recommendations. This would also allow us to compare the performance of the $\Delta_{(i)}/GI/1$ to the more common $GI/GI/1$ model in various application contexts.




\section*{Appendix}



\subsection*{Proof of Lemma~\ref{lem:mama95}.}
  Rewrite $\ty_n$ as
\(
\ty_n = (\Psi(\sn x + y_n) - \Psi(\sn x + y)) - (\Psi(\sn x + y) - \sn \Psi(x)).
\)
Now, using the fact that the Skorokhod reflection map is Lipschitz continuous under the uniform metric (see Lemma 13.4.1 and Theorem 13.4.1 of \cite{Wh01}) we have
\(
(\Psi(\sn x + y_n) - \Psi(\sn x + y)) \leq \| y_n - y \|,
\)
where $\| \cdot \|$ is the uniform metric. It follows that
\(
\ty_n \leq \| y_n - y \| + (\Psi(\sn x + y) - \sn \Psi(x)),
\)
Now, by Theorem 9.5.1 of \cite{Wh01b} we know that as $n \rightarrow \infty$
\[
(\Psi(\sn x + y) - \sn \Psi(x)) \stackrel{a.s.}{\rightarrow} \ty, \text{ in } (\sD_{\lim},M_1).
\]
Using this result, and the fact that by hypothesis $y_n$ converges to $y$ in $(\sD_{lim},J_1)$ we have
\(
\ty_n \stackrel{a.s.}{\rightarrow} \ty, \text{ in } (\sD_{\lim},M_1).
\)
\EndProof

\subsection{Proof of Lemma \ref{lem:fluid-queue-intensity}.}
First, suppose $\bar{Q}(t) > 0$.  It follows that
\(
\bar F(t) -\m t >  \inf_{-T_0 \leq s \leq t} (\bar F(s) - \mu s) = w
\)
where the latter equality follows because the queue starts empty at time $0$, and the fluid netput is positive before time $0$ (Note that we ignore the positive part operator in the definition of
$\Psi$, as the systems starts empty at time $-T_0$). Now, let $t^* = \sup \{0 \leq s \leq t | (\bar F(s) - \mu s) =  \inf_{0 \leq s \leq t} (\bar F(s) - \mu s)\}$ be the point at which the infimum is achieved, on the right hand side. It follows that
\(
\bar F(t) - \m t > \bar F(t^*) - \m t^*,
\)
in turn yielding
\[
\rho(t) = \sup_{0 \leq s \leq t} \frac{\bar F(t) - \bar F(s)}{\m (t-s)} > 1.
\]

Next, suppose $\bar{Q}(t) = 0$,
  $\bar{X}(t) = \Psi(\bar{X})(t)$ and there exists an $r < t$ such that
  $\Psi(\bar{X})(t) = \Psi(\bar{X})(s)$ for all $s \in [r,t]$. It
follows that
\(
\bar F(t) - \m t = - \sup_{-T_0 \leq s \leq t} (-(\bar F(s) - \m s)),
\)
implying there exists a point $r^* \in [0,t]$ such that
\(
\bar F(t) - \m t = \bar F(r^*) - \m r^*.
\)
This, in turn, implies that
\[
\sup_{0 \leq s \leq t} \frac{\bar F(t) - \bar F(s)}{\m (t - r)} \geq \frac{\bar F(t) -
  \bar F(r^*)}{\m (t - r^*)} = 1.
\]
However a simple contradiction argument shows that
\[
\sup_{0 \leq s \leq t} \frac{\bar F(t) - \bar F(s)}{\m (t - r)} > 1
\]
is impossible, implying that
\[
\sup_{0 \leq s \leq t} \frac{\bar F(t) - \bar F(s)}{\m (t - r)} = 1.
\]

Finally, consider case (iii). We have,
$\forall r < t$,
\[
-(\bar F(t) - \m t) = \sup_{-T_0 \leq s \leq t} (-(\bar F(s) - \m s)) > \sup_{-T_0 \leq s \leq r}
(-(\bar F(s) - \m s)).
\]
It follows that
\(
-(\bar F(t) - \m t) > -(\bar F(r) - \m r),
\)
implying
\[
1 > \frac{\bar F(t) - \bar F(r)}{\m (t - r)} ~ \forall r \in
[0,t).
\]
\EndProof

\subsection{Proof of Theorem \ref{thm:dist-approx}.} \label{sec:dist-approx-appendix}
\noindent (i) \textbf{Overloaded regime.}\\
\Proof
First, note that $\t$ is the first instant of an
end of overloading phase, and the current overloaded phase ends at
$\t$. In the overloaded state $\bar{Q}(t) > 0$, implying
that $\Psi(\bar{X})(t)$ is a constant. Using the definition of
$\nabla_t^{\bar{X}}$ it follows that $\Psi(\bar{X})(t) =
-\bar{X}(t^*)$, and $\bar{Q}(t) = \bar{X}(t) - \bar{X}(t^*) = (\bar F(t) -
\bar F(t^*) - \m(t - t^*))$. Next,
from Theorem \ref{thm:queue-length-diffusion}, it is obvious that
\(
\frac{Q^n(t)}{\sn} \stackrel{d}{\approx} \tZ^n_t.
\)

Next, from Remark 1 after Lemma \ref{lem:X-hat}, $\hat{X}(t) -
\hat{X}(t^*) = \int_{t^*}^t \sqrt{g^{'}(s)} dW_s$, which can be seen
to be a diffusion process that starts from $0$ at $t^*$. Noting
that $\nabla_t^{\bar{X}}$ does not change on the interval $(t^*,\t)$,
it follows that $X^* = \sup_{s \in \nabla_{t^*}^{\bar X}} \{-\hat{X}(s))$ is a fixed
random variable, and $\tZ^n_t$ has an initial condition $\tZ^n_{t^*} =
\hat{X}(t^*) - X^*$. It is straightforward to see that $\tZ_n^{\cdot}$
is the strong solution to the mentioned SDE, since it is adapted to
the filtration generated by $W$.
\EndProof

\noindent (ii) \textbf{Underloaded regime.}\\
This result is immediate from the definition of the limit processes.\\
\\
\noindent (iii) \textbf{Middle- and End-of critically-loaded state.}\\
\Proof
For any $t \in (t^*,\t)$ we have $\bar{Q}(t) = 0$. From the weak convergence result in Theorem \ref{thm:queue-length-diffusion} we have
\(
Q^n(t) \stackrel{d}{\approx} n \bar{Q}(t) + \sn \hat{Q}(t),
\)
and expanding the definition of $\hat{Q}$ it follows that
\(
Q^n(t) \stackrel{d}{\approx} \sn (\hat{X}(s) + \sup_{s \in \nabla_t^{\bar{X}}} (-\hat{X}(s))).
\)
Using the fact that $\Psi(\bar{X})(t) = w = -\bar{X}(t) ~\forall~ t \in (t^*,\t)$ in a critically loaded regime, it follows that $\nabla_t^{\bar{X}} = (t^*,t]$ for $t \in (t^*,\t)$. Thus, we have
\(
Q^n(t) \stackrel{d}{\approx} \sn(\hat{X}(s) + \sup_{t^* < s \leq t} (-\hat{X}(s))).
\)
Let $u = t - t^*$. Then, after a change of variables we obtain
\(
Q^n(u+t^*) \stackrel{d}{\approx} \sn (\hat{X}(u+t^*) + \sup_{0 \leq s < u} (-\hat{X}(s))).
\)

Since $W^0$ is a Brownian Bridge process, the strong Markov property of Brownian motion (\cite{KaSh91}) implies that $\hat{X}(u+t^*) -\hat{X}(t^*) = \hat{X}(u)$. Substituting this into the expression above we see that we have,
\(
Q^n(u+t^*) = Q^n(u) + \hat{X}(t^*),
\)
where $\hat{X}(t^*)$ is the starting state of the process in the
middle-of-critically loaded state. A simple change of variables will
provide the desired result. A similar argument will hold for the end-of-critical loading state as well.
\EndProof

\noindent (iv) \textbf{End of Overloading state.}\\
\Proof
By definition for any $\t >0$, $t - \frac{\t}{\sn}$ is a point of
overloading. Therefore 
\[
\frac{Q^n(t-\frac{\t}{\sn})}{\sn} = \hat{X}^n(t - \frac{\t}{\sn}) + \sn (F(t -
\frac{\t}{\sn}) - \mu(t - \frac{\t}{\sn}) ) + \Psi(\hat{X}^n + \sn
\bar{X})(t - \frac{\t}{\sn}) - \sn \Psi(\bar X)(t - \frac{\t}{\sn}).
\]
Without loss of generality, we assume that service started when the
queue was in the overloaded state, so that $\Psi(\bar X)(t -
\frac{\t}{\sn}) = 0$. Now, using the fact the derivative $f$ exists, the mean value
theorem implies the existence of a point $\tilde{t} \in [t-\frac{\t}{\sn}, t]$ such that
\(
F(t - \frac{\t}{\sn}) = F(t) - f(\tilde{t})\frac{\t}{\sn}.
\)
Adding and subtracting the term $f(t) \t/\sn$ to the expression above we have
\[
F(t - \frac{\t}{\sn}) = F(t) - f(t) \frac{\t}{\sn} + f(t) \frac{\t}{\sn} - f(\tilde{t})\frac{\t}{\sn}.
\]
Substituting this into the expression for $Q^n$ above, and introducing the term $\hat{X}^n(t)$, we obtain
\[
\begin{split}\frac{Q^n(t-\frac{\t}{\sn})}{\sn} = \hat{X}^n(t &- \frac{\t}{\sn}) - \hat{X}^n(t) +
\hat{X}^n(t)+ \sn (F(t) - \mu t) -(f(t) - \mu) \t\\ + &\Psi(\hat{X}^n + \sn \bar{X})(t - \frac{\t}{\sn}) + (f(t) - f(\tilde{t})) \frac{\t}{\sn}. \end{split}
\]
Now, using Lemma \ref{lem:X-hat} and the continuity of the limit
process we see that $\hat{X}^n(t - \frac{\t}{\sn}) - \hat{X}^n(t)
\Rightarrow 0$. Further, since $f$ is bounded by virtue of being defined on a finite interval we have $\t (f(t) - f(\tilde{t}))/\sn \rightarrow \infty$ as $n \rightarrow \infty$. Next, consider the term
\(
\hat{Z}(t) := \hat{X}^n(t) + \sn(F(t) - \mu t) + \Psi(\hat{X}^n + \sn \bar{X})(t - \frac{t}{\sn}).
\)
Let $\d > 0$ be sufficiently small, so that the following decomposition of the
expression above holds,
\[
\begin{split} \hat{Z}^n(t) = \sup_{-T_0 \leq s < t-\d} (\hat{X}^n(t) &+ \sn(F(t) -
  \mu t) - \hat{X}^n(s) - \sn \bar{X}(s) )\\ &\vee \sup_{t-\d \leq s \leq
    t - \frac{\t}{\sn}} (\hat{X}^n(t) + \sn(F(t) -
  \mu t) - \hat{X}^n(s) - \sn \bar{X}(s)). \end{split}
\]
Let $t^* = \sup \{\nabla_t^{\bar{X}} \backslash \{t\} \}$. Consider
the first term on the RHS above, and call it $\hat{Z}^n_1(t)$. Since the queue is
overloaded before $t$ no points are ``added'' to the correspondence $\nabla_t^{\bar{X}}$, it follows from the definition of an end of
overloading point that $(F(t) - \mu t) = - \Psi(\bar{X})(t) \equiv - \Psi(\bar{X})(t^*+\d)$. This, in turn, provides
\(
\hat{Z}_1^n(t) = \hat{X}^n(t) + \Psi(\hat{X}+\sn \bar{X})(t^*+\d) - \sn \Psi(\bar{X})(t^*+\d).
\)
Using Lemma \ref{lem:mama95}, it follows that $\hat{Z}^n_1(t)
\Rightarrow \hat{X}(t) + \sup_{s \in \nabla_t^{\bar{X}} \backslash \{t\}} (
-\hat{X}(s))$ as $n \rightarrow \infty$, followed by letting $\d \to 0$.
Next, consider the second term
\begin{eqnarray*}
\hat{Z}_2^n(t) &=& \sup_{t-\d \leq s \leq t - \frac{\t}{\sn}} ( \hat{X}^n(t) + \sn(F(t) -
  \mu t) - \hat{X}^n(s) - \sn \bar{X}(s) )\\
&\leq& \sup_{t - \d
  \leq s \leq t - \frac{\t}{\sn}} (\hat{X}^n(t) - \hat{X}^n(s)) +
\sup_{t - \d \leq s \leq t - \frac{\t}{\sn}} \sn (\bar{X}(t) -
\bar{X}(s))\\
&\leq& \sup_{t - \d  \leq s \leq t}  (\hat{X}^n(t) - \hat{X}^n(s)) +
\sup_{t - \d \leq s \leq t - \frac{\t}{\sn}} \sn (\bar{X}(t) -
\bar{X}(s)).
\end{eqnarray*}
For large $n$, as the queue is overloaded at $t-\frac{\t}{\sn}$ it
follows that
\[
\hat{Z}_2^n(t)  \leq \sup_{t-\d \leq s \leq t} (\hat{X}(t) -
\hat{X}(s)) + \sn (\bar{X}(t) - \bar{X}(t - \frac{\t}{\sn})).
\]
Again, by the mean value theorem 
\begin{eqnarray*}
\sn (\bar{X}(t) - \bar{X}(t - \frac{\t}{\sn})) &=& \sn (F(t) - F(t -
\frac{\t}{\sn}) - \m \frac{\t}{\sn})\\
&=& \sn (f(t) - \m) \frac{\t}{\sn} + (f(t) - f(\tilde{t})) \t,
\end{eqnarray*}
where $\tilde{t} \in [t-\frac{\t}{\sn},t]$. Since, $\tilde{t} \to t$
as $n \to \infty$, by the (right) continuity of $f$ it follows that $ f(t) - f(\tilde{t}) \to 0$ as $n \to \infty$. Then it follows by an application of Lemma \ref{lem:X-hat} (and using the Skorokhod's
almost sure representation) that
\(
\overline{\lim}_{n \rightarrow \infty} \hat{Z}^n_2(t) \leq \hat{X}(t) +
\sup_{t - \d \leq s \leq t} (-\hat{X}(s)) + (f(t) - \m) \t.
\)
On the other hand, for a lower bound, using the mean value theorem again, we have
\(
\hat{Z}^n_2(t) \geq \hat{X}^n(t) - \hat{X}^n(t - \frac{\t}{\sn}) +
(f(t) - \m) \t + (f(t) - f(\tilde{t})) \t.
\)
Once again, using the continuity of $f$, the almost sure representation theorem and Lemma \ref{lem:X-hat}, and noting the continuity of the limit process $\hat{X}$, we have
\[
\underline{\lim}_{n \rightarrow \infty} \hat{Z}^n_2(t) \geq (f(t) - \m)
\t ~ a.s.
\]
Now, using the limits derived for $\hat{Z}^n_1 \text{ and
  } \hat{Z}^n_2$ it follows that
\begin{eqnarray*}
\frac{Q^n(t - \frac{\t}{\sn})}{\sn} &\Longrightarrow& -(f(t) - \m) \t +
\sup_{s \in \nabla_t^{{\bar{X}}} \backslash \{t\}} (\hat{X}(t) - \hat{X}(s))\vee
( f(t) - \mu) \t\\
&=& \left( \hat{X}(t) +
\sup_{s \in \nabla_t^{\bar{X}} \backslash \{t\}} (-\hat{X}(s)) - (f(t) - \m)\t \right)^+.
\end{eqnarray*}
\EndProof

\subsection*{Proof of Proposition \ref{thm:sample-path}.}
The proof is a consequence of the following lemma, which consolidates
Lemmas 6.5, 6.6 and 6.7 in \cite{MaMa95}. The lemma characterizes the
points of discontinuity (and continuity) of the process $\tY(t) =
\sup_{s \in \nabla_t^{\bar{X}}}(-\hat{X}(s))$ in relation to the
correspondence $\nabla_t^{\bar{X}}$. We do not prove these conditions,
but direct the interested reader to \cite{MaMa95}.

\begin{lemma} \label{lem:continuity-conditions}
~A point $t \in [-T_0,\infty)$ is characterized as follows.\\
\noindent \emph{(i)} Continuity Conditions.\\
The following are equivalent:
\begin{enumerate}
\item $t$ is a continuity point.
\item $t \in \nabla_t^{\bar{X}} = \{t\}$, \emph{or} $t \not \in
  \nabla_t^{\bar{X}}$, \emph{or} $t \in \nabla_t^{\bar{X}} \not = \{t\}$ and
  $t$ is not isolated in $\nabla_t^{\bar{X}}$ and $\nabla_t^{\bar{X}}
  \subseteq \nabla_u^{\bar{X}} ~ for ~ some~ u > t$.
\end{enumerate}

\noindent \emph{(ii)} Right-discontinuity Conditions.\\
The following are equivalent:
\begin{enumerate}
\item $t$ is a point of right-discontinuity.
\item $t \in \nabla_t^{\bar{X}} \not = \{t\}$ and $\nabla_u^{\bar{X}}
  \subseteq (t,u] ~ \forall ~ u > r$.
\item $\tY(t) = \tY(t-) > \tY(t+) = -\hat{X}(t)$.
\end{enumerate}

\noindent \emph{(iii)} Left-discontinuity Conditions.\\
The following are equivalent:
\begin{enumerate}
\item $t$ is a point of left-discontinuity.
\item $t \in \nabla_t^{\bar{X}} \not = \{t\}$ and $t$ is isolated in
  $\nabla_t^{\bar{X}}$.
\item $\tY(t) = \tY(t+) = -\hat{X}(t) > \tY(t-)$.
\end{enumerate}
\end{lemma}
A point of right-discontinuity can be seen to be left-continuous,
coupled with an ordering on the right and left limits, such that
$\tY(t-) > \tY(t+)$. Similarly, a point of left-discontinuity is
right-continuous, and the limits are ordered such that $\tY(t+) >
\tY(t-)$. Using these definitions, we proceed to prove the upper-semicontinuity of the limit process.\\

\Proof [Proposition \ref{thm:sample-path}]
By definition, $\hat{X}$ is continuous,
and it suffices to check that a sample path of the component $\tY(t) = \sup_{s \in
  \nabla_t^{\bar{X}}} ( -\hat{X}(s) )$ is upper-semicontinuous. To
see this, consider the pullback of the level set $\tY^{-1}[a,\infty) =
\{t \in [-T_0,\infty) | \tY(t) \geq a \}$. It suffices to check that
this is a closed set; see \cite{Ru06}. Let $\{\t_n\} \subseteq \{t \in [-T_0,\infty) |
\tY(t) \geq a \}$ be a sequence of points such that
\(
\t_n \rightarrow \t
\)
as $n \rightarrow \infty$, where $\t \in [-T_0,\infty)$ is an
arbitrary point in the domain of $\tY$. Thus, if $\e > 0$, then there
exists an $n_0 \in \mathbb{N}$ such that $\forall$ $n \geq n_0$, $\e
\geq \t - \t_n \geq -\e$. If $\t$ is a continuity point, then the
conclusion is obvious. On the other hand, suppose that $\t$ is a
left-discontinuity point. By part (iii) of Lemma \ref{lem:continuity-conditions} it
follows that $\tY(\t-) < \tY(\t+) = \tY(\t)$. By the definition of a
left-discontinuity there exits an interval $[t^*,\t)$, where $t^* =
\sup \nabla_{\t}^{\bar{X}} \backslash \{\t\}$, on which $\tY$ is (locally)
continuous. Fix $\d > 0$, then there exists an $\eta > 0$ such that if
$ \t$\mbox{\tiny $-$}   $- t \geq -\eta$, then $\d \geq \tY(\t-) - \tY(t) \geq
-\d$. If $\e$ is small enough, then there exists $n_0$
such that $\forall ~\, n \geq n_0$, $\t$\mbox{\tiny $-$}  $-\t_n > -\eta$. It follows
that $\d \geq \tY(\t_n) - \tY(\t-) \geq a - \tY(\t-)$, implying that
\(
\tY(\t-) \geq a - \d.
\)
Since $\d$ is arbitrary, it follows that $\tY(\t-) \geq a$, in turn
implying that $\tY(\t) \geq 0$. Thus, $\t \in \tY^{-1}[a,\infty)$.

Next, suppose that $\t$ is a right-discontinuity point. Then, from part (ii)
of Lemma \ref{lem:continuity-conditions} we have $\tY(\t) = \tY(\t-) <
\tY(\t+)$. Furthermore, for any $u > \t$, we have $\nabla_u^{\bar{X}}
  \subseteq (\t,u]$ implying that these are continuity points (by
  part (i) of Lemma \ref{lem:continuity-conditions}). Using an argument similar to
  that for a left-discontinuity, on points to the right of $\t$, it follows that
  $\tY(\t) \geq a$. This implies that the pullback set
  $\tY^{-1}[a,\infty)$ is closed. As $\{\t_n\}$ is an arbitrary
  sequence in $\tY^{-1}[a,\infty)$ it is necessarily true that
  $\tY$ is upper-semicontinuous.
\EndProof

\subsection*{Proof of Corollary \ref{cor:uniform-arrivals}.}
The proof of the corollary depends on Lemma \ref{lem:continuity-conditions} above.\\

\Proof [Corollary \ref{cor:uniform-arrivals}]
Recall that $\hat{Q} = \hat{X} + \tY$, where $\tY(t) = \sup_{s \in \nabla_t^{\bar{X}}} (-\hat{X}(s))$. The proof of (i) follows directly from part (i) of Lemma \ref{lem:continuity-conditions}. Next, recall from the proof of Corollary \ref{cor:uniform-queue-length-diffusion} that $\nabla_{\t}^{\bar{X}} = \{-T_0,\t\}$. Thus, $\t$ is isolated in the set and it follows that part (iii) of Lemma \ref{lem:continuity-conditions} is satisfied. On the other hand, recall that $\nabla_t^{\bar{X}} = \{t\} \subset (\t,t], ~ \forall t > \t$, and $\t$ can also be a point of right-discontinuity, by part (ii) of Lemma \ref{lem:continuity-conditions}. Thus, $\t$ is one or the other depending on the path of $\hat{X}$. If $\hat{X}(\t) < 0$ then $\tY(\t+) = \tY(\t) > \tY(\t-)$ and $\t$ is a point of left-discontinuity. Otherwise, if $\hat{X}(\t) \geq 0$, then s$\tY(\t) = \tY(\t-) = 0 > \tY(\t+)$ and $\t$ is a point of right-discontinuity.
\EndProof

\bibliographystyle{unsrt} 
\bibliography{refs-queueing}
\end{document}